\documentclass[11pt]{amsart}
\usepackage{amsmath, amssymb, amscd, mathrsfs, url, pinlabel,verbatim}
\usepackage[pagebackref]{hyperref}
\usepackage[margin=1in,marginparwidth=0.75in,centering,letterpaper,dvips]{geometry}
\usepackage{color,dcpic,latexsym,graphicx,epstopdf,comment}
\usepackage[all]{xy}
\usepackage[dvipsnames]{xcolor}
\usepackage{tikz,tikz-cd,pgfplots}
\usepackage{upquote,tabularx,textcomp}
\usepackage[shortlabels]{enumitem}
\usepackage{mathtools}
\usepackage[color=blue!20!white,textsize=tiny]{todonotes}

\title{Ribbon concordance and fibered predecessors, II: the general case}

\author[John A. Baldwin]{John A. Baldwin}
\address{Department of Mathematics \\ Boston College}
\email{john.baldwin@bc.edu}

\author[Jonathan Hanselman]{Jonathan Hanselman}
\address{Department of Mathematics \\ Indiana University}
\email{jonahans@iu.edu}

\author[Steven Sivek]{Steven Sivek}
\address{Department of Mathematics\\Imperial College London}
\email{s.sivek@imperial.ac.uk}

\makeatletter
\newtheorem*{rep@theorem}{\rep@title}
\newcommand{\newreptheorem}[2]{%
\newenvironment{rep#1}[1]{%
 \def\rep@title{#2 \ref{##1}}%
 \begin{rep@theorem}}%
 {\end{rep@theorem}}}
\makeatother

\newtheorem {theorem}{Theorem}
\newreptheorem{theorem}{Theorem}
\newtheorem {lemma}[theorem]{Lemma}
\newtheorem {proposition}[theorem]{Proposition}
\newtheorem {corollary}[theorem]{Corollary}
\newtheorem {conjecture}[theorem]{Conjecture}

\newtheorem {question}[theorem]{Question}

\numberwithin{equation}{section}
\numberwithin{theorem}{section}

\theoremstyle{definition}
\newtheorem{definition}[theorem]{Definition}

\newtheorem{remark}[theorem]{Remark}
\newtheorem*{remark*}{Remark}

\newenvironment{usetheoremcounterof}[1]{%
  \renewcommand\thetheorem{\ref{#1}}\theorem}{\endtheorem\addtocounter{theorem}{-1}}

\newlist{pcases}{enumerate}{1}
\setlist[pcases]{
  label=\bf{Case~\arabic*:}\protect\thiscase.~,
  ref=\arabic*,
  align=left,
  labelsep=0pt,
  leftmargin=0pt,
  labelwidth=0pt,
  parsep=0pt
}
\newcommand{\case}[1][]{%
  \if\relax\detokenize{#1}\relax
    \def\thiscase{}%
  \else
    \def\thiscase{~#1}%
  \fi
  \item
}

\newcommand{\Mor}{\mathrm{Mor}}

\newcommand{\Z}{\mathbb{Z}}

\newcommand{\R}{\mathbb{R}}

\newcommand{\F}{\mathbb{F}}
\newcommand{\Q}{\mathbb{Q}}

\renewcommand{\phi}{\varphi}

\newcommand\hfk{\mathit{HFK}}
\newcommand\hfkhat{\widehat{\hfk}}

\newcommand\HFhat{\widehat{\mathit{HF}}}
\newcommand\HFK{\widehat{\mathit{HFK}}}

\newcommand{\inr}{\operatorname{int}}

\newcommand{\HF}{\mathit{HF}}

\newcommand{\gr}{\operatorname{gr}}

\newcommand{\pt}{\mathrm{pt}}

\newcommand{\CFD}{\widehat{\mathit{CFD}}}

\newcommand{\Alg}{\mathcal{A}}

\usetikzlibrary{calc,intersections}
\tikzset{every picture/.style=thick}
\tikzset{link/.style = { white, double = black, line width = 1.75pt, double distance = 1.25pt, looseness=1.75 }}
\tikzset{crossing/.style = {draw, circle, dotted, minimum size=0.5cm, inner sep=0, outer sep=0}}
\pgfplotsset{compat=1.12}

\begin{document}

\begin{abstract}
The first and third authors recently proved that for each  knot $K\subset S^3$ there are only finitely many hyperbolic fibered knots which are ribbon concordant to $K$. In this paper, we remove the \emph{hyperbolic} constraint, proving that every knot in $S^3$ has only finitely many fibered predecessors under ribbon concordance. The key new input is an inequality relating the knot Floer homology of a generalized satellite knot with that of its companion, proved via the immersed curves formulation of bordered Heegaard Floer homology, which should be  of independent interest. Our work, together with results of Kojima--McShane, also leads to an explicit upper bound on the Gromov norm of the complement of any fibered predecessor of a knot $K \subset S^3$, in terms of the arc index and genus of $K$.
\end{abstract}

\maketitle

\section{Introduction}
\label{sec:intro}

Agol proved in \cite{agol-ribbon} that ribbon concordance defines a partial order on the set of knots in $S^3$, affirming a conjecture of Gordon from \cite{gordon-ribbon}; we therefore write $J\leq K$ to mean that a knot $J$ is ribbon concordant to the knot $K$. In the same paper, Gordon made the still-open conjecture  that there is no infinite descending chain of distinct ribbon concordant knots, $\dots \leq K_3 \leq K_2\leq K_1$. In \cite{bs-ribbon}, the first and third authors proposed the \emph{a priori} stronger conjecture that each knot in $S^3$ has only finitely many predecessors under ribbon concordance:

\begin{conjecture}
\label{conj:pred}
For each knot $K\subset S^3$, there are only finitely many knots $J\leq K$.
\end{conjecture}

In support of this conjecture, they proved that each knot in $S^3$ has only finitely many hyperbolic fibered predecessors \cite[Theorem 1.2]{bs-ribbon}. The main goal of this paper is to remove the hyperbolic constraint, and prove this finiteness result for \emph{all} fibered predecessors:

\begin{theorem} \label{thm:main}
For each knot $K\subset S^3$, there are only finitely many fibered knots $J\leq K$.
\end{theorem}

Since ribbon predecessors of fibered knots are fibered \cite{silver,kochloukova}, this proves Conjecture \ref{conj:pred} for all fibered knots:

\begin{corollary}
For each fibered knot $K\subset S^3$, there are only finitely many knots $J\leq K$.\qed
\end{corollary}

The proof of \cite[Theorem 1.2]{bs-ribbon} first used Floer homology to bound the dilatation of the pseudo-Anosov monodromy of any hyperbolic fibered $J\leq K$, as well as the genus of $J$, and then appealed to the fact that there are only finitely many conjugacy classes of pseudo-Anosov homeomorphisms of a given surface with dilatation less than a fixed constant. If $J$ is not hyperbolic then its monodromy is either periodic or reducible. The periodic case is easy to deal with. The bulk of this paper deals with the reducible case. At heart, our strategy once again involves bounding the dilatations of the pseudo-Anosov components of the reducible monodromy, though there are serious obstacles to doing so by the methods in \cite{bs-ribbon}, as explained in \cite[\S 3]{bs-ribbon} and reiterated in \S\ref{ssec:proof} below. The key new input which allows us to bypass these difficulties is an inequality relating the knot Floer homology of a generalized satellite knot (see \S\ref{sec:generalized-satellite}) with that of its companion:

\begin{theorem}
\label{thm:gensat} Let $K\subset Y$ be a generalized satellite knot with companion $C \subset Z$, whose pattern $P$ has winding number $w\geq 1$. Then \[\dim\HFK(Z,C)\leq \dim\HFK(Y,K).\]
\end{theorem}

Theorem \ref{thm:gensat} is stated in much greater generality than is needed for our application, and should be of independent interest. It yields in particular the corollary below for satellite knots in $S^3$. Even this special case was previously unknown, though Shen proved it for $(1,1)$-patterns in \cite{shen-satellite}:

\begin{corollary}
\label{cor:simplesatellite}
Let $K\subset S^3$ be a  satellite knot with companion $C$, whose pattern $P\subset S^1\times D^2$ has winding number $w\geq 1$. Then \[\dim\HFK(S^3,C)\leq \dim\HFK(S^3,K).\]
\end{corollary}

We prove Theorem \ref{thm:gensat} via the immersed curves formulation of bordered Heegaard Floer homology due to the second author, Rasmussen, and Watson in \cite{hrw}. We will ultimately apply it to lifts of fibered satellite knots in their cyclic branched covers (see Corollary~\ref{cor:satcover}), as indicated in \S\ref{ssec:proof}.

Finally, the techniques that go into our proof of Theorem \ref{thm:main}, together with results of Kojima and McShane relating the entropy of a pseudo-Anosov homeomorphism and the volume of its mapping torus \cite{kojima-mcshane}, allow us to prove that the arc index and genus of a knot $K\subset S^3$ give rise to an upper bound on the Gromov norm $\lVert \cdot \rVert$ of the complement of any fibered predecessor of $K$. This generalizes a similar result in \cite[Theorem 1.4]{bs-ribbon} bounding volume for \emph{hyperbolic} fibered predecessors. Let \[v_3= 1.014941606...\] denote the volume of the regular ideal tetrahedron in $\mathbb{H}^3$. Then we have the following:

\begin{theorem}\label{thm:volume}
Suppose that $K\subset S^3$ is a knot of genus $g$ and arc index $\delta$. Then \[\lVert S^3\setminus J\rVert \leq \tfrac{3\pi}{v_3}(2g-1)\log(\delta!)\] for every fibered knot $J\leq K$.
\end{theorem}

\subsection{On the proof}\label{ssec:proof}

Let $K$ be a knot in $S^3$. As mentioned above, the main challenge in proving Theorem \ref{thm:main} is showing that there are finitely many fibered knots $J\leq K$ with reducible monodromy. Any such $J$ is a satellite knot whose pattern has winding number $w\geq 1$. Let us assume for example that $J$ is the $(1,6)$-cable of a hyperbolic (necessarily fibered) knot $C$, and explain why there are only finitely many such $J\leq K$, as even this very simple case illustrates some of the key ideas.

Let $h:\Sigma\to \Sigma$ be the monodromy of $J$.  The reducing set $\Gamma$ of $h$ cuts $\Sigma$ into seven pieces, \[\overline{\Sigma \setminus \Gamma}  = \Sigma_0 \cup S_0 \cup S_1 \cup \dots \cup S_5,\] where $\Sigma_0$ contains $\partial \Sigma$, and each $S_i$ is a copy of the fiber surface of $C$. Moreover $h$ sends each $S_i$ to $S_{i+1}$, where the subscripts are understood mod 6, and we may assume that the restriction \[h^6|_{S_0}:S_0\to S_0\] is freely isotopic to the pseudo-Anosov monodromy of $C$, which we denote by $\varphi$.

Our aim is  to prove that there are only finitely many possibilities for  $C$. This would follow as in the proof of \cite[Theorem~1.2]{bs-ribbon} from an a priori upper bound on the dilatation  of $\varphi$, the idea being that (1) $C$ is determined by its complement, which is determined by the conjugacy class of $\varphi$; (2) up to conjugation, there are only finitely many pseudo-Anosov maps of a given surface with dilatation less than some fixed constant; and (3) $g(C) < g(J)\leq g(K)$, the second inequality following for fibered $J$ from a theorem of Gilmer \cite{gilmer}.

Let $J_n\subset \Sigma_n(J)$ denote the lift of $J$ in its $n$-fold cyclic branched cover. Then $J_n$ is a fibered knot with monodromy $h^n$. The knot Floer homology of $J_n$ bounds the Nielsen number of its monodromy, \[N(h^n) \leq \dim \HFK(\Sigma_n(J),J_n),\] as shown by Ni \cite{ni-fixed} and Ghiggini--Spano \cite{ghiggini-spano}. It follows from this and the discussion above that when $n=6k$, the number of fixed points of $\varphi^k$ is bounded as follows, \[\#\mathrm{Fix}(\varphi^k) \leq N(h^{6k}) \leq \dim \HFK(\Sigma_{6k}(J),J_{6k}).\]
Together with the relationship between the dilatation of a pseudo-Anosov map and the fixed points of its iterates (cf.~\cite[Lemma 2.4]{bs-ribbon}), this implies that the dilatation $\lambda(\varphi)$ satisfies  \begin{equation}\label{eq:inequality-3k}\lambda(\varphi)  = \lim_{k\to \infty}\big( \#\mathrm{Fix}(\varphi^k)\big)^{\frac{1}{k}} \leq \liminf_{k\to \infty} \big(\dim \HFK(\Sigma_{6k}(J),J_{6k})\big )^{\frac{1}{k}}.\end{equation}  It would therefore suffice for our aim to bound the right hand side in terms of $K$, and this is the most natural approach following the work in \cite{bs-ribbon}.

The first and third authors proved in \cite[Lemma~2.1]{bs-ribbon} that
\begin{equation} \label{eqn:delta}
\dim \HFK(\Sigma_{n}(K),K_{n}) \leq (\delta!)^n,
\end{equation}
where $\delta$ is the arc index of $K$. Now, the ribbon concordance from $J$ to $K$ lifts to an annulus in the $n$-fold cyclic branched cover of $S^3\times I$ branched along the concordance. If the complement of this annulus is a ribbon $\mathbb{Z}/2$-homology cobordism from the complement of $J_n$ to the complement  of $K_{n}$ then the work of Daemi--Lidman--Vela-Vick--Wong \cite{dlvvw} combined with \eqref{eqn:delta} implies that
\begin{equation} \label{eq:inequality-intro}
\dim \HFK(\Sigma_{n}(J),J_{n})\leq \dim \HFK(\Sigma_{n}(K),K_{n}) \leq (\delta!)^n.
\end{equation}
We would like to apply this in the case $n=6k$, and conclude from \eqref{eq:inequality-3k} that \[\lambda(\varphi) \leq (\delta!)^6,\] which would complete the proof that there are only finitely many possibilities for $C$. The problem with this approach, alluded to in \cite[\S3]{bs-ribbon}, is that while the annulus complement above is always a ribbon cobordism, it need not be a $\Z/2$-homology cobordism or even a $\Q$-homology cobordism for \emph{any} $n=6k$, much less for a sequence of such $n$ tending to infinity.  

We use Theorem \ref{thm:gensat} to bypass this issue as follows. First, we show that the annulus complement above is a ribbon $\Z/2$-homology cobordism from the complement of $J_n$ to that of $K_n$ for an infinite increasing sequence of primes \[n=p_1, p_2, p_3,\dots,\] which  implies  the inequality \eqref{eq:inequality-intro} for all such $n$ (see Theorems \ref{thm:hfk-cover} and \ref{thm:branched-cover-ribbon}). Importantly, we may then assume that these $p_j$ are relatively prime to 6, in which case $J_n$ is a generalized satellite knot with companion $C_n\subset \Sigma_n(C)$, whose pattern has winding number 6. Theorem~\ref{thm:gensat} then implies that 
\begin{equation} \label{eq:inequality-intro2}
\dim \HFK(\Sigma_{n}(C),C_{n})\leq \dim \HFK(\Sigma_{n}(J),J_{n})
\end{equation}
for all $n=p_j$. Note that $C_n$ is fibered with pseudo-Anosov monodromy $\varphi^n$. Combining \eqref{eq:inequality-intro2} and \eqref{eq:inequality-intro} with the relationship between knot Floer homology and fixed points  then gives the inequality
\[ \big(\#\mathrm{Fix}(\varphi^n)\big)^{\frac{1}{n}}\leq \big(\dim \HFK(\Sigma_{n}(C),C_{n})\big)^{\frac{1}{n}}\leq \delta ! \]
for all $n=p_j$. Finally, taking the limit as $j$ goes to infinity yields the bound \[\lambda(\varphi) \leq \delta!,\] which proves the desired finiteness of fibered $J\leq K$ which are $(1,6)$-cables of hyperbolic knots. 

Now, general fibered satellite knots are much more complicated than such cables. For instance,  the restriction of the monodromy $h$ to the outermost component $\Sigma_0$ can also be pseudo-Anosov; the components of $\overline{\Sigma\setminus \Sigma_0}$ can   intersect the reducing system $\Gamma$ in their interiors; and these components can form more than one orbit under $h$. Establishing finiteness in the general case thus proceeds by an inductive argument and is attended by new subtleties, but the spirit of the proof is similar to what we have sketched above.

\subsection{A technical result and a question} 
The induction alluded to above really goes into proving our main technical theorem below, which we then use to prove Theorems \ref{thm:main} and \ref{thm:volume}:

\begin{theorem} \label{thm:hfk-j-finite}
Given natural numbers $M$ and $g$, there are only finitely many genus-$g$ fibered knots $K\subset S^3$ which satisfy \[\dim\HFK(\Sigma_{n}(K),K_{n})\leq M^{n}\] for infinitely many primes $n$. Moreover, \[\lVert S^3\setminus K \rVert \leq \tfrac{3\pi}{v_3}(2g-1)\log(M)\] for any such knot $K$.

\end{theorem}

It is natural to ask whether this theorem also holds for nonfibered knots:

\begin{question}
Does Theorem \ref{thm:hfk-j-finite} hold without the \emph{fibered} hypothesis?
\end{question}

Floer homologists may find this question interesting for its own sake. But its primary importance for us is that if the finiteness result in Theorem \ref{thm:hfk-j-finite} holds without the fibered hypothesis, then one can prove Conjecture \ref{conj:pred} in full generality, and hence Gordon's original conjecture about descending chains of ribbon concordant knots, following our proof of Theorem \ref{thm:main} from Theorem \ref{thm:hfk-j-finite}.
 
\subsection{Concluding remarks} First, just as noted in \cite{bs-ribbon}, Theorem \ref{thm:main} holds by the same arguments if ribbon concordance is replaced by \emph{strong homotopy-ribbon concordance}, since the results we use about ribbon concordance from \cite{dlvvw} hold in this more general setting. Second, since it is possible to enumerate the conjugacy classes of pseudo-Anosov homeomorphisms of a given surface with dilatation less than a fixed constant, our proof of Theorem \ref{thm:main} should in principle lead to an algorithm which for any knot $K\subset S^3$ produces an explicit finite set of knots $\mathcal{S}$ such that every fibered knot $J\leq K$ is contained in $\mathcal{S}$. We will leave this for others to work out in detail.

\subsection{Conventions} All Floer homology groups in this paper are with coefficients in $\F:=\Z/2$.

\subsection{Organization}
In \S\ref{sec:cyclic-covers}, we study the question of which cyclic covers of ribbon concordance complements are  $\Z/2$-homology cobordisms, and use this to establish inequalities between the knot Floer homologies of the lifts of ribbon concordant knots in $S^3$ to  their cyclic branched covers. In \S\ref{sec:generalized-satellite}, we provide background and basic results about generalized satellite knots. In \S\ref{sec:reducing-curves}, we prove several results about the reducible monodromies of fibered satellite knots in $S^3$. In \S\ref{sec:satellite-inequality}, we use the immersed curves formulation of bordered Heegaard Floer homology to prove the inequality in Theorem \ref{thm:gensat}. We finally put everything together in \S\ref{sec:main-proof} to prove Theorem \ref{thm:hfk-j-finite} and then Theorems \ref{thm:main} and \ref{thm:volume}.

\subsection{Acknowledgements}
JAB was supported by NSF Grant DMS-2506250 and a Simons Fellowship. SS was supported by the Engineering and Physical Sciences Research Council [grant number UKRI1016].  No data were created or analyzed in this work. We thank Qiuyu Ren and Liam Watson for interesting and helpful conversations.

\section{Branched cyclic covers and knot Floer homology} \label{sec:cyclic-covers}
\newcommand\resultant{\operatorname{Res}}

Given a knot $K\subset S^3$ and a natural number $n$, we denote by $K_n$ the lift of $K$ to its $n$-fold cyclic branched cover $\Sigma_n(K)$, both here and throughout this paper. Supposing that $J\leq K$, an important ingredient in our proof of Theorem \ref{thm:main}, as discussed in \S\ref{ssec:proof}, is the following inequality relating the knot Floer homologies of $J_n$ and $K_n$ for certain $n$:

\begin{theorem}
\label{thm:hfk-cover}
For each knot $K\subset S^3$ there is a finite set $S$ of prime numbers such that 
\[\dim \HFK(\Sigma_n(J),J_n) \leq \dim\HFK(\Sigma_n(K), K_n)\] for any knot $J\leq K$ and every natural number $n$ which is not a multiple of any element  of $S$.
\end{theorem}

The first and third authors proved this inequality for $n$ which are powers of $2$ in  \cite[Corollary 2.3]{bs-ribbon}, but we need the stronger result above to prove Theorem \ref{thm:main} for non-hyperbolic $J$.  In particular, if $J$ is a satellite with winding number $w$ then we need the above inequality for an infinite sequence of $n$ which are all relatively prime to $w$, as the outline in \S\ref{ssec:proof} suggests.  Theorem \ref{thm:hfk-cover} follows from the behavior of Floer homology under ribbon $\Z/2$-homology cobordisms  \cite{dlvvw}, together with the $p=2$ case of the result below, which is a strengthening of \cite[Lemma 2.2]{bs-ribbon}:

\begin{theorem} \label{thm:branched-cover-ribbon}
Let $C \subset S^3 \times I$ be a ribbon concordance from $J$ to $K$. Let \[W_n:\Sigma_n(J)\to\Sigma_n(K)\] be the $n$-fold branched cyclic cover of $S^3\times I$ branched along $C$, and let $C_n$ be the lift of $C$ to $W_n$.  Given a prime $p$, there is a finite set $S$ of primes depending only on $K$ and $p$ such that $W_n \setminus \nu(C_n)$ is a ribbon $\Z/p$-homology cobordism whenever $n$ is not a multiple of any element of  $S$.
\end{theorem}

The construction of $S$ (denoted by $S_{K,p}$ in Proposition~\ref{prop:resultants} below) is quite explicit: it is determined  by $p$ and the mod $p$ reduction of the Alexander polynomial $\Delta_K(t)$, and the product of its elements is at most $p^{2\deg \Delta_K(t)}$.  We omit these details from Theorem~\ref{thm:branched-cover-ribbon} because it will suffice for our purposes just to know that some $S$ exists, but these features may be useful to others in the future.

To prove Theorem~\ref{thm:branched-cover-ribbon}, we first determine when the boundary components of the branched cover $W_n$ are $\Z/p$-homology spheres.  Fox showed that for any knot $K \subset S^3$ and integer $n \geq 1$ we have
\[ \left| H_1(\Sigma_n(K); \Z) \right| = \left| \prod_{j=0}^{n-1} \Delta_K(e^{2\pi i j / n}) \right|, \]
where the left side should be interpreted as zero if $H_1(\Sigma_n(K))$ is infinite; see, for example, \cite[Theorem~8.21]{burde-zieschang}.  If we write
\[ \tilde\Delta_K(t) = t^{\deg \Delta_K} \Delta_K(t) \in \Z[t], \]
then the above is equivalent to
\[ \left| H_1(\Sigma_n(K); \Z) \right| = \left| \prod_{\zeta^n = 1} \tilde\Delta_K(\zeta) \right| = \left| \resultant_\Z(t^n-1, \tilde\Delta_K(t)) \right|, \]
where we write $\resultant_R(f,g) \in R$ for the resultant of two polynomials $f,g \in R[t]$.  We will use this formula to prove the following, in which $\F_p$ and $\overline{\F}_p$ are respectively the finite field of order $p$ and its algebraic closure.

\begin{lemma} \label{lem:resultant-nth-root}
Given a knot $K \subset S^3$ and a prime $p$, the  branched cyclic cover $\Sigma_n(K)$ is a $\Z/p$-homology sphere if and only if no root of $\tilde\Delta_K(t)$ in $\overline{\F}_p$ is also a root of $t^n-1$.
\end{lemma}

\begin{proof}
We know that $\Sigma_n(K)$ is a $\Z/p$-homology sphere if and only if the order of $H_1(\Sigma_n(K); \Z)$ is not a multiple of $p$.  Since
\begin{align*}
\resultant_{\overline{\F}_p}(t^n-1,\tilde\Delta_K(t)) &= \resultant_{\F_p}(t^n-1,\tilde\Delta_K(t)) \\
&\equiv \resultant_\Z(t^n-1,\tilde\Delta_K(t)) \pmod{p}, 
\end{align*}
this is equivalent to showing that
\[ \resultant_{\overline{\F}_p}(t^n-1,\tilde\Delta_K(t)) \neq 0. \]
Now, the resultant of two polynomials over a field vanishes if and only if those polynomials have a common factor of positive degree, so \[\resultant_{\overline\F_p}(t^n-1,\tilde\Delta_K(t)) = 0\] if and only if some root of $\tilde\Delta_K(t)$ is a root of $t^n-1$ in $\overline{\F}_p$.
\end{proof}

\begin{corollary} \label{cor:ribbon-outgoing-homology-sphere}
Suppose that $J\leq K$ are knots in $S^3$ and  $p$ is prime. If $\Sigma_n(K)$ is a $\Z/p$-homology sphere then so is $\Sigma_n(J)$.
\end{corollary}

\begin{proof}
By Lemma~\ref{lem:resultant-nth-root}, no root $\omega \in \overline{\F}_p$ of $\tilde\Delta_K(t)$ satisfies $\omega^n = 1$.  Since $J$ is ribbon concordant to $K$, a theorem of Gilmer \cite{gilmer} says that $\Delta_J(t)$ divides $\Delta_K(t)$, so any root $\zeta \in \overline{\F}_p$ of $\tilde\Delta_J(t)$ is also a root of $\tilde\Delta_K(t)$ and hence satisfies $\zeta^n \neq 1$.  Another appeal to Lemma~\ref{lem:resultant-nth-root} then allows us conclude that $\Sigma_n(J)$ is also a $\Z/p$-homology sphere.
\end{proof}

We  now describe the construction and properties of the set $S=S_{K,p}$ which  features in Theorem \ref{thm:branched-cover-ribbon}, as well as in Theorem \ref{thm:hfk-cover} in the case when $p=2$:

\begin{proposition} \label{prop:resultants}
For each knot $K \subset S^3$ and prime $p$, there is a finite set $S_{K,p}$ of primes with the following properties:
\begin{enumerate}
\item $\Sigma_n(K)$ is a $\Z/p$-homology sphere if $n$ is not a multiple of any element of $ S_{K,p}$,  \label{i:SKp-branched-cover}
\item $p \not\in S_{K,p}$, \label{i:p-not-in-SKp}
\item $S_{K,p}$ depends only on $p$ and the mod $p$ reduction of $\Delta_K(t)$, \label{i:SKp-depends-on-Delta}
\item if $\Delta_J(t)$ divides $\Delta_K(t)$ then $S_{J,p} \subset S_{K,p}$, \label{i:SKp-divisibility}
\item the product of all elements of $S_{K,p}$ satisfies \label{i:SKp-product}
\[ \prod_{q\in S_{K,p}} q \leq p^{2\deg \Delta_K(t)} \leq p^{2g(K)}. \]
\end{enumerate}
\end{proposition}

\begin{proof}
Lemma~\ref{lem:resultant-nth-root} tells us that $\Sigma_n(K)$ is a $\Z/p$-homology sphere if and only if no root of $\tilde\Delta_K(t)$ in $\overline{\F}_p$ has multiplicative order dividing $n$.  We will therefore find an integer that is simultaneously a multiple of the multiplicative orders of all nonzero roots of $\tilde\Delta_K(t)$, and we will take $S_{K,p}$ to be the set of its prime divisors.

We factor $\tilde\Delta_K(t)$ into irreducibles in $\F_p[t]$ as
\[ \tilde\Delta_K(t) = t^a \cdot f_1(t) \cdot f_2(t) \cdot \ldots \cdot f_k(t), \]
where $f_j(t) \neq t$ and hence $f_j(0) \neq 0$ for all $j$; note that $a \geq 0$ is only positive if the leading coefficient of $\Delta_K(t)$ is zero mod $p$.  Let $d_j = \deg f_j(t) \geq 1$ for each $j=1,\dots,k$.  Then the roots of $f_j(t)$ have degree $d_j$ over $\F_p$, meaning that they live in the unique degree-$d_j$ extension $\F_{p^{d_j}}$ of $\F_p$.  Since the multiplicative group $\F_{p^{d_j}}^\times$ is finite of order $p^{d_j}-1$, we see that each root of $f_j(t)$ in $\overline{\F}_p$ has finite multiplicative order dividing $p^{d_j} - 1$.  We thus define
\[ S_{K,p} = \left\{ q \text{ prime} \,\middle\vert\, q \text{ divides } \prod_{j=1}^k (p^{d_j} - 1) \right\}. \]
It is immediate from the definition that $S_{K,p}$ is finite, that $p\not\in S_{K,p}$, and that $S_{K,p}$ is determined entirely by $p$ and the mod $p$ factorization of $\Delta_K(t)$; these last two claims are properties~\eqref{i:p-not-in-SKp} and \eqref{i:SKp-depends-on-Delta} of the proposition.  Property~\eqref{i:SKp-divisibility} is also immediate, because if $\Delta_J(t)$ divides $\Delta_K(t)$ then any irreducible factor of $\tilde\Delta_J(t)$ in $\F_p[t]$ is also an irreducible factor of $\tilde\Delta_K(t)$.

To verify property~\eqref{i:SKp-branched-cover}, we suppose that $n$ is not a multiple of any $q \in S_{K,p}$.  Given any root $\omega \in \overline{\F}_p$ of $\tilde\Delta_K(t)$, say of multiplicative order $m \geq 1$, we must have $m > 1$, because $\Delta_K(1)=1$ implies that $\omega \neq 1$.  Thus $m$ has at least one prime divisor $r$, and this $r$ must belong to $S_{K,p}$: indeed, $\omega$ is a root of one of the irreducible factors $f_j(t)$, so its order $m$ divides the corresponding $p^{d_j}-1$, whence $r$ divides $p^{d_j}-1$ as well.  But then $n$ cannot be a multiple of $m$ since it is not a multiple of $r$, and so $\omega^n - 1 \neq 0$.  We have now shown that no root of $\tilde\Delta_K(t)$ in $\overline{\F}_p$ is simultaneously a root of $t^n-1$, so it follows as explained above that $\Sigma_n(K)$ is a $\Z/p$-homology sphere.

This leaves only property~\eqref{i:SKp-product}, for which we observe by construction that $\prod_{q \in S_{K,p}} q$ divides the product $\prod_{j=1}^k (p^{d_j}-1)$.  This latter product satisfies
\[ \prod_{j=1}^k (p^{d_j}-1) \leq \prod_{j=1}^k p^{d_j} = p^{\sum_j d_j} \leq p^{\deg \tilde\Delta_K(t)} = p^{2\deg\Delta_K(t)}, \]
so the product of all $q \in S_{K,p}$ is bounded above by $p^{2\deg\Delta_K(t)} \leq p^{2g(K)}$ as claimed.
\end{proof}

\begin{proof}[Proof of Theorem~\ref{thm:branched-cover-ribbon}]
Let $S = S_{K,p}$ be the finite set of primes provided by Proposition~\ref{prop:resultants}.  If $n$ is not a multiple of any element of $S$ then $\Sigma_n(K)$ is a $\Z/p$-homology sphere, and then Corollary~\ref{cor:ribbon-outgoing-homology-sphere} guarantees that $\Sigma_n(J)$ is as well.  This implies that
\[ H_i(\Sigma_n(J) \setminus \nu(J_n); \F_p) \cong H_i(\Sigma_n(K) \setminus \nu(K_n); \F_p) \cong
\begin{cases} \F_p & i=0,1 \\ 0 & i \geq 2, \end{cases} \]
where in each case the group $H_1 \cong \F_p$ is generated by the corresponding meridian $\mu_{J_n}$ or $\mu_{K_n}$.

Following the proof of \cite[Lemma 2.2]{bs-ribbon}, we note that $S^3 \times I \setminus \nu(C)$ can be built from $S^3 \setminus \nu(J)$ by attaching $4$-dimensional $1$- and $2$-handles since $C$ is a ribbon concordance, and that there are exactly as many $1$-handles as $2$-handles since $S^3 \times I \setminus \nu(C)$ has Euler characteristic zero.  This lifts to a handle decomposition of the $n$-fold cyclic cover $W_n \setminus \nu(C_n)$, in which we attach equal numbers of $1$- and $2$-handles to $\Sigma_n(J) \setminus \nu(J_n)$.  This shows that \[W_n \setminus \nu(C_n):\Sigma_n(J) \setminus \nu(J_n) \to \Sigma_n(K) \setminus \nu(K_n)\] is a ribbon cobordism, with Euler characteristic
\begin{equation} \label{eq:chi-wn}
\chi(W_n \setminus \nu(C_n)) = n \cdot \chi(S^3 \times I \setminus \nu(C)) = 0,
\end{equation}
and that
\begin{equation} \label{eq:h1-wn-positive}
\dim_{\F_p} H_1(W_n \setminus \nu(C_n); \F_p) \geq \dim_{\F_p} H_1(\Sigma_n(J) \setminus \nu(J_n); \F_p) = 1.
\end{equation}
Flipping this upside down, we can also build $W_n \setminus \nu(C_n)$ by attaching equal numbers of $3$- and $2$-handles to $\Sigma_n(K) \setminus \nu(K_n)$, which implies that the map
\begin{equation} \label{eq:h1-surjection-wn}
\F_p \cong H_1(\Sigma_n(K) \setminus \nu(K_n); \F_p) \to H_1(W_n \setminus \nu(C_n); \F_p)
\end{equation}
induced by inclusion is surjective.  We combine \eqref{eq:h1-wn-positive} and \eqref{eq:h1-surjection-wn} to get
\[ H_1(W_n \setminus \nu(C_n); \F_p) \cong \F_p, \]
which implies in turn that the surjection \eqref{eq:h1-surjection-wn} is an isomorphism.  We then conclude that the inclusion-induced map
\begin{equation} \label{eq:h1-injection-wn}
\F_p \cong H_1(\Sigma_n(J) \setminus \nu(J_n); \F_p) \to H_1(W_n \setminus \nu(C_n); \F_p) \cong \F_p
\end{equation}
is likewise an isomorphism, because it is injective: the meridian $\mu_{J_n}$ that generates the homology on the left is homologous in $W_n \setminus \nu(C_n)$ to the meridian $\mu_{K_n}$, which in turn generates the right side  since \eqref{eq:h1-surjection-wn} is an isomorphism.

Now for all integers $i \geq 3$ we have \[H_i(W_n \setminus \nu(C_n); \F_p) = 0,\] because $W_n \setminus \nu(C_n)$ can be built from $\Sigma_n(J) \setminus \nu(J_n)$ using handles of index at most $2$ and $H_i(\Sigma_n(J) \setminus \nu(J_n); \F_p) = 0$.  Having determined that
\[ \dim_{\F_p} H_i(W_n \setminus \nu(C_n); \F_p) = \begin{cases} 1 & i=0,1 \\ 0 & i \geq 3, \end{cases} \]
it now follows from the Euler characteristic computation \eqref{eq:chi-wn} that
\[ H_2(W_n \setminus \nu(C_n); \F_p) = 0 \]
as well.  Since the mod $p$ homology groups of $W_n \setminus \nu(C_n)$ and of the knot complements $\Sigma_n(J) \setminus \nu(J_n)$ and $\Sigma_n(K) \setminus \nu(K_n)$ vanish in all degrees greater than $1$, and since \eqref{eq:h1-surjection-wn} and \eqref{eq:h1-injection-wn} are isomorphisms, we can at last conclude that $W_n \setminus \nu(C_n)$ is indeed an $\Z/p$-homology cobordism.
\end{proof}

\begin{proof}[Proof of Theorem \ref{thm:hfk-cover}]
The $p=2$ case of Theorem \ref{thm:branched-cover-ribbon} says that for each knot $K\subset S^3$ there is a finite set $S$ of primes such that $\Sigma_n(J)\setminus \nu(J_n)$ is ribbon $\Z/2$-homology cobordant to $\Sigma_n(K)\setminus \nu(K_n)$ for any knot $J\leq K$ and every  $n$ which is not  a multiple of any $q\in S$. But the fact that these knot complements are ribbon $\Z/2$-homology cobordant implies that \[\dim \HFK(\Sigma_n(J),J_n) \leq \dim\HFK(\Sigma_n(K), K_n),\] by the work of Daemi, Lidman, Vela-Vick, and Wong in \cite[Corollary~4.13]{dlvvw}.\end{proof}

\section{Generalized satellite knots}
\label{sec:generalized-satellite}

In proving Theorem \ref{thm:main}, we will study lifts $J_n\subset \Sigma_n(J)$ of satellite  knots $J\leq K$, in cases where $\Sigma_n(J)$ is a rational homology sphere. Such lifts are examples of what we call \emph{generalized satellite knots}. In this section, we provide background and  results about generalized satellites which we will use to prove the inequality in Theorem \ref{thm:gensat} and ultimately Theorem \ref{thm:main}. The results here are not especially novel, but we could not find a satisfying reference for many of them.

Let $M$ be a compact orientable 3-manifold with torus boundary, which satisfies $H_2(M) = 0$.  Then the long exact sequence on homology associated with the pair $(M,\partial M)$ gives an exact sequence
\[ 0 \to H_2(M,\partial M; \Z) \xrightarrow{\partial_*} \underbrace{H_1(\partial M; \Z)}_{\cong \,\Z^2} \xrightarrow{i_*} H_1(M; \Z). \]
The map $i_*$ has rank 1 over $\Q$ by half lives half dies (see, for example, \cite[Lemma 3.5]{hatcher-3mfld}); hence, so does $\partial_*$.  This implies that \[H_2(M,\partial M ;\Z) \cong \Z,\] since it injects into a free abelian group with image of rank 1.  
Let $S \subset M$ be a properly embedded, oriented surface such that $[S]$ generates $H_2(M,\partial M;\Z)$. Then we can write
\[ \partial_*[S] = [\partial S] = n \cdot \alpha_M \]
for some integer $n \geq 1$ and some primitive class \[\alpha_M \in H_1(\partial M; \Z).\]  Note that the integer $n$ is well-defined and the pair $([S],\alpha_M)$ is well-defined up to an overall sign. We may now make the following two definitions in terms of the notation above.

\begin{definition}\label{def:rational-longitude}
The class $\alpha_M$ is called the \emph{rational longitude} of $M$. We will often view it instead as an  unoriented curve in the torus $\partial M$, where it is well-defined up to isotopy.
\end{definition}

\begin{definition}\label{def:winding-number}
Given a knot $P\subset M$, oriented so that the algebraic intersection number $[S] \cdot [P]$ is nonnegative, we define the \emph{winding number} of $P$ to be the rational number \[ w = \frac{[S] \cdot [P]}{n} \geq 0. \] Note that this quantity is independent of $S$.
\end{definition}

\begin{remark}\label{rmk:relation-S}
In the definition above, we may assume that $P$ is transverse to $S$, and consider its intersection with the exterior \[E_P = M \setminus \nu(P).\]  If $\mu_P \subset \partial E_P$ is the meridian of $P$, then $S \cap E_P$ gives rise to a relation
\begin{equation} \label{eq:rational-longitude-vs-muP}
[\partial S]=n \alpha_M = nw \cdot [\mu_P]
\end{equation}
in $H_1(E_P;\Z)$. If $\alpha_M$ is actually nullhomologous in $M$, as opposed to just rationally so, then $n=1$ and this relation simplifies to $\alpha_M = w\cdot [\mu_P]$; notably, this applies whenever $M$ is a solid torus.
\end{remark}

We may now give the main definition of this section.

\begin{definition}\label{def:generalized-satellite}
A \emph{generalized satellite knot} is a knot $K$ in a rational homology 3-sphere $Y$ whose exterior
\[ E_K = Y \setminus \nu(K) \]
is irreducible, has incompressible boundary, and contains an embedded incompressible torus $T$ that is not boundary-parallel.  We may therefore write
\[ Y = M_1 \cup_T M_2, \qquad \partial M_1 = T = -\partial M_2,\qquad K \subset \inr(M_2). \]
The fact that $H_2(Y;\Z)=0$ implies via the Mayer--Vietoris sequence that
\begin{equation} \label{eq:h2-satellite}
H_2(M_1;\Z) = H_2(M_2;\Z) = 0,
\end{equation}
and we further require as part of this definition that the rational longitudes $\alpha_{M_1}, \alpha_{M_2} \subset T$ have pairwise distance one. We  refer to $T$ as the \emph{satellite torus.}
\end{definition}

Just as for ordinary satellite knots in $S^3$, there are  natural notions of \emph{pattern} and \emph{companion} for generalized satellite knots. For the two definitions below, let $K \subset Y$ be a generalized satellite knot with satellite torus $T \subset Y$ and  corresponding decomposition \[Y = M_1\cup_T M_2,\] as in Definition \ref{def:generalized-satellite}. 

\begin{definition} \label{def:generalized-pattern}
The \emph{pattern} is the knot $P \subset M_2$ whose image under the inclusion $M_2 \hookrightarrow Y$ is $K$.
\end{definition}

\begin{definition} \label{def:generalized-companion}
Let $\mu_C \subset \partial M_1$ be the curve identified with $\alpha_{M_2} \subset \partial M_2$ in $T$, and let
\[ Z = M_1(\mu_C) \]
be the closed 3-manifold obtained by Dehn filling $M_1$ along $\mu_C$.  The \emph{companion knot} \[C \subset Z\] is given by the core of the filling solid torus.
\end{definition}

\begin{remark}\label{rmk:dual}
The companion $C$ has exterior $M_1$, and in the torus $T = \partial M_1$ its rational longitude $\lambda_C = \alpha_{M_1}$ is geometrically dual to its meridian $\mu_C = \alpha_{M_2}$, by Definition \ref{def:generalized-satellite}.
\end{remark}

As mentioned at the top, we will be especially interested in generalized satellite knots that are the lifts of satellite knots in $S^3$ to their branched cyclic covers. Given a knot $K\subset S^3$, let us denote by $K_n$ the lift of $K$ to its $n$-fold branched cyclic cover $\Sigma_n(K)$, as usual. We note the following:

\begin{proposition} \label{prop:branched-cover-satellite}
Let $K \subset S^3$ be a  satellite knot with companion $C \subset S^3$, whose pattern has winding number $w \geq 1$, and let $n$ be a natural number  relatively prime to $w$.  If $\Sigma_n(K)$ is a rational homology sphere, then $K_n \subset \Sigma_n(K)$ is a generalized satellite knot with companion $C_n \subset \Sigma_n(C)$, whose pattern also has winding number $w$.
\end{proposition}

\begin{proof}
Since $K$ is a satellite with companion $C$, we can write \[S^3 = M_1\cup_T M_2,\] where $T$ is the satellite torus, \[M_1 = S^3\setminus \nu(C), \quad \textrm{and}\quad M_2\cong S^1\times D^2\] is a solid torus containing the pattern knot $P$. The claim is that since $n$ is relatively prime to $w$, this decomposition lifts to a generalized satellite decomposition for $K_n \subset\Sigma_n(K)$. 

Note that the branched covering $\Sigma_n(K)\to S^3$ restricts to the $n$-fold cyclic covering \[\Sigma_n(K)\setminus \nu(K_n)\to S^3\setminus \nu(K),\] corresponding to the kernel of the homomorphism \[\rho:\pi_1(S^3\setminus \nu(K)) \to \Z/n\] which sends the meridian $\mu_K=\mu_P$ to a generator. We have that \[[\mu_C] :=[\mu_{M_2}]= w[\mu_P]\] by \eqref{eq:rational-longitude-vs-muP}, so the fact that $n$ is relatively prime to $w$ implies that $\rho$ also sends $\mu_C$ to a generator of $\Z/n$. It follows that the preimage \[T_n\subset \Sigma_n(K)\] of $T$ is still a single torus, and that $T_n$ separates $\Sigma_n(K)$ into two pieces, \[\Sigma_n(K) = (M_1)_n \cup_{T_n} (M_2)_n,\] where the first piece is the $n$-fold cyclic cover of $S^3\setminus \nu(C)$, or equivalently, \[(M_1)_n = \Sigma_n(C)\setminus \nu(C_n),\] and the second piece is the $n$-fold branched cyclic cover of $M_2$ along $P$. We must check that this decomposition satisfies the requirements in the definition of a generalized satellite. 

First, note that  $\Sigma_n(K)$ is a rational homology sphere by assumption. 

Second, since $S^3\setminus \nu(K)$ is irreducible with incompressible torus boundary, the same is true of its $n$-fold cyclic cover, which is precisely the knot complement $\Sigma_n(K)\setminus \nu(K_n)$.  Moreover, $T_n$ is incompressible and non-boundary-parallel in this complement since it's the lift of the incompressible, non-boundary-parallel torus $T$ under the cyclic covering. For the fact that irreducibility is inherited by the cover, see, for example, \cite[Theorem 3.15]{hatcher-3mfld}. The fact that the incompressibility of these surfaces is preserved under lifting to the cover follows easily from the fact that a 2-sided surface is incompressible if and only if it is $\pi_1$-injective; see, for example, \cite[Corollary 3.3]{hatcher-3mfld}.

Third, the rational longitude $\alpha_{(M_1)_n}$ is just the Seifert longitude $\lambda_{C_n}$, while the rational longitude $\alpha_{(M_2)_n}$ is simply the meridian $\mu_{C_n}$. These curves are dual on $T_n$ as desired.

It follows that $K_n\subset\Sigma_n(K)$ is a generalized satellite knot with companion $C_n\subset \Sigma_n(C)$.

For the last claim of the proposition, let $P_n\subset (M_2)_n$ be the lift of $P\subset M_2$. Let $S$ be a meridional disk in $M_2$  transverse to $P$  whose algebraic intersection with $P$ is nonnegative. Then the winding number is given by this algebraic intersection number, \[w=[S]\cdot[P],\] by Definition \ref{def:winding-number}. Now, let $S_n$ be the lift of $S$ in $(M_2)_n$, and observe that \[\partial S_n = \mu_{C_n} = \alpha_{(M_2)_n}.\] Since each point on $P$ has a unique preimage in the branched cover, and the covering map preserves local orientations, the winding number of $P_n$ is  given by \[[S_n]\cdot [P_n] = [S] \cdot [P] = w\] as well.
\end{proof}

The lemma below describes how the meridian and Seifert longitude of a nullhomologous generalized satellite knot are related to the meridian and (rational) longitude of its companion. 

\begin{lemma} \label{lem:winding-number-satellite}
Let $K \subset Y$ be a nullhomologous generalized satellite knot, with satellite torus $T \subset Y$ and corresponding decomposition \[Y = M_1 \cup_T M_2.\]  Let $P \subset M_2$ and $C \subset Z$ be the associated pattern and companion knots.  Let \[E_K = Y \setminus \nu(K), \qquad E_P = M_2\cap E_K = M_2 \setminus \nu(P),\] with $\mu_K$ and $\lambda_K$ denoting the meridian and Seifert longitude of $K$ in each exterior.
If $P$ has winding number $w \geq 0$, then $w$ is an integer and any Seifert surface $\Sigma$ for $K$ meeting $T$ transversely satisfies
\[ [\Sigma \cap T] = w[\lambda_C] \]
as elements of $H_1(T;\Z)$.  In particular, this gives us relations $[P] = w[\lambda_C]$ in $H_1(M_2;\Z)$, and
\begin{align*}
[\lambda_K] &= w[\lambda_C], \\
nw[\mu_K] &= n[\mu_C]
\end{align*}
in $H_1(E_P;\Z)$, where $n \geq 1$ is the order of the rational longitude $\alpha_{M_2}$.
\end{lemma}

\begin{proof}
Let $\Sigma \subset E_K$ be a Seifert surface for $K$ which is transverse to $T$.  Let $k \geq 1$ denote the order of the rational longitude $[\lambda_C] = \alpha_{M_1}$ in $H_1(M_1;\Z)$. 
Observe that $\Sigma \cap M_1$ is a properly embedded surface in $M_1$, so its boundary class
\begin{equation*} \label{eq:boundary-sigma-y0}
\partial_*[\Sigma \cap M_1] \in H_1(\partial M_1;\Z) = H_1(T;\Z)
\end{equation*}
is in the kernel of the inclusion-induced map \[H_1(T;\Z)\to H_1(M_1;\Z),\]  and is thus an integer multiple
\begin{equation}\label{eq:boundary-mk}
\partial_*[\Sigma \cap M_1]= km[\lambda_C]\end{equation}
of $k$ times the rational longitude.  Note that \[\Sigma \cap  E_P\subset M_2\] extends to the core of the tubular neighborhood $\nu(P)$ to have boundary $P$, producing a relation
\begin{equation} \label{eq:P-lambdaC}
[P] = km[\lambda_C]
\end{equation}
in $H_1(M_2;\Z)$.

Now, the meridian $\mu_C$ and rational longitude $\lambda_C$ are  dual to each other in $T$, as noted in Remark \ref{rmk:dual}, and $\mu_C$ is identified in $T$ with the rational longitude $\alpha_{M_2}$ of $M_2$, by Definition \ref{def:generalized-companion}.  Letting $S$ be the surface generating 
\[H_2(M_2,\partial M_2;\Z) \cong \Z,\] with $\partial_*[S] = n\cdot \alpha_{M_2}$ and $n \geq 1$, we deduce that the intersection pairing on $M_2$ then satisfies
\begin{equation} \label{eq:S-lambdaC}
[S] \cdot [\lambda_C] = n,
\end{equation}
as can be seen by pushing $\lambda_C$ slightly into $M_2$ so that it meets $S$ in a single point near each copy of $\alpha_{M_2}$ in $\partial S$.  Applying \eqref{eq:S-lambdaC} and then \eqref{eq:P-lambdaC}, we have
\[ kmn = km\left([S] \cdot [\lambda_C]\right) = [S] \cdot km[\lambda_C] = [S] \cdot [P] = nw, \]
so finally $w = km$.  In particular $w \in \Z$, so the relations \eqref{eq:boundary-mk} and \eqref{eq:P-lambdaC} become
\[ [\Sigma \cap T] = \partial_*[\Sigma \cap M_1] = km [\lambda_C]=w[\lambda_C] \]
in $H_1(T;\Z)$ and \[[P] = w[\lambda_C]\] in $H_1(M_2;\Z)$.  Similarly, the intersection $\Sigma \cap E_P \subset M_2$  gives rise to the relation \[[\lambda_K] = w[\lambda_C]\] in $H_1(E_P;\Z)$.  Finally, the relation \[nw[\mu_K] = [\mu_C]\] is precisely \eqref{eq:rational-longitude-vs-muP} after we identify $\mu_K = \mu_P$ and $\mu_C = \alpha_{M_2}$.
\end{proof}

Our final lemma generalizes a well-known result about  fibered satellite knots in $S^3$. We will not need it in full generality, but  include it  since the proof is short, in case it is useful to others.

\begin{lemma} \label{lem:fibered-satellite}
If $K \subset Y$ is a fibered generalized satellite knot, then its associated pattern  has  positive winding number.
\end{lemma}

\begin{proof}
Let $E_K = Y\setminus \nu(K)$, and let $F$ be a fiber of the fibration $E_K \to S^1$. The homomorphism
\[ \phi: \pi_1(E_K) \twoheadrightarrow \Z \]
given by $\gamma \mapsto [F] \cdot [\gamma]$ defines an infinite cyclic cover of $E_K$ homeomorphic to $F \times \R$.  In particular, $\ker(\phi)$ is a free group of rank $2g(F)$.

If $P$ has winding number $w=0$, then Lemma~\ref{lem:winding-number-satellite} says that $[P] = 0$ in $H_1(M_2;\Z)$, so in fact $P$ has a Seifert surface in $M_2$, whose image in $Y$ is a Seifert surface $F'$ disjoint from the satellite torus $T$.  Then $H_2(Y)=0$ implies $H_2(E_K) = 0$, which in turn implies the injectivity of the map
\[ \partial_*: H_2(E_K, \partial E_K) \to H_1(\partial E_K). \]
This map sends both $[F]$ and $[F']$ to the longitude $[\lambda_K]$, so in fact \[[F] = [F']\in H_2(E_K,\partial E_K).\]
We identify $\pi_1(T) \cong \Z\oplus\Z$ as a subgroup of $\pi_1(E_K)$, since $T$ is incompressible.  Then for any class $\gamma \subset \pi_1(T)$ we have
\[ \phi(\gamma) = [F] \cdot [\gamma] = [F'] \cdot [\gamma] = 0 \]
since $F'$ and $T$ are disjoint, so $\pi_1(T) \subset \ker(\phi)$.  But this means that we have found a $\Z\oplus\Z$ subgroup of a free group, which is impossible, so we must have had $w > 0$ after all.
\end{proof}

\section{Fibered satellite knots and reducible monodromy} \label{sec:reducing-curves}

Let $h: \Sigma \to \Sigma$ be the monodromy of a fibered knot $K$ in a closed, oriented 3-manifold; that is, $\Sigma$ is a compact, connected, oriented surface with one boundary component, and $h$ is a homeomorphism restricting to the identity on $\partial \Sigma$. 
Thurston's classification of surface homeomorphisms \cite{thurston-surfaces} says that $h$ is freely isotopic to a homeomorphism $\phi:\Sigma\to \Sigma$ in \emph{Nielsen--Thurston form}, meaning that:
\begin{itemize}
\item there exists a possibly empty embedded multicurve $\Gamma \subset \Sigma$ which is fixed setwise by $\varphi$;
\item and if $S$ is a component of $\Sigma \setminus \Gamma$ and $n$ is the smallest positive integer such that $\varphi^n(S)=S$, then $\phi^n|_S$ is freely isotopic to either a periodic or a pseudo-Anosov homeomorphism.
\end{itemize}
We will assume that $\Gamma$ is minimal with respect to inclusion, in which case it is unique up to isotopy \cite[Theorem C]{blm}.  Finally, we denote by \[\Sigma_0 \subset \Sigma\] the closure of the component of $\Sigma \setminus \Gamma$ containing  $\partial\Sigma$, and refer to $\Sigma_0$ as the \emph{outermost component}.

We will be most interested in the case that $K$ is a fibered knot in $S^3$; in this case, $K$ is a satellite knot if and only if the reducing system for its monodromy is nonempty \cite{thurston-fiber}.  (In particular, there are no boundary-parallel or nullhomotopic components because $\Gamma$ is assumed minimal.)
Our goal in this section is to clarify the relationship between the monodromy of $K$ and its reducing system on one hand, and the structure of $K$ as a satellite knot on the other. We begin with two technical lemmas.

\begin{lemma} \label{lem:separating-multicurve}
Let $K \subset Y$ be a fibered knot in a rational homology 3-sphere whose monodromy is freely isotopic to a map $\phi: \Sigma \to \Sigma$.  If $c \subset \Sigma$ is a nonempty embedded multicurve that is fixed setwise by $\phi$, then $c$ separates $\Sigma$.
\end{lemma}

\begin{proof}
Fix a component $c_0$ of $c$.  Then the $\phi$-orbit of $c_0$ is a subset of $c$, and the suspension of this orbit is an embedded torus in the mapping torus $M_\phi \cong Y \setminus \nu(K)$; the image of this torus after filling $K$ back in is a torus $T \subset Y$.  If  $c$ does not separate $\Sigma$, then we can find a path in $\Sigma \setminus c$ from one side of $c_0$ to the other, and hence a closed, embedded curve $\alpha \subset \Sigma$ that meets $c$ transversely in a single point of $c_0$.  Since $\alpha$ avoids the other components of the $\phi$-orbit of $c_0$, its image under the inclusion of a fiber
\[ \Sigma \hookrightarrow Y \setminus \nu(K) \hookrightarrow Y \]
is a closed curve $\alpha' \subset Y$ that meets $T$ transversely in a single point.  But then the intersection form
\[ H_1(Y;\Z) \times H_2(Y;\Z) \to \Z \]
is nonzero since $[\alpha'] \cdot [T] = \pm1$, and this contradicts the assumption that $H_1(Y;\Q) = 0$.
\end{proof}

\begin{lemma} \label{lem:outermost-separating}
Let $K \subset S^3$ be a fibered knot whose monodromy is freely isotopic to a map $\phi: \Sigma \to \Sigma$ in Nielsen--Thurston form, with reducing system $\Gamma$.   Then 
\[ \Gamma \cap \Sigma_0 = \partial\Sigma_0 \setminus \partial\Sigma, \]
and each component of  $\partial\Sigma_0 \setminus \partial\Sigma$ separates $\Sigma$.
\end{lemma}

\begin{proof}
The map $\phi$ fixes both $\Gamma$ and $\partial\Sigma$ setwise, so it must fix the outermost component $\Sigma_0$.  This means that $\phi$ also fixes $\Gamma \cap \Sigma_0$ setwise.  Given any component $\gamma \subset \Gamma \cap \Sigma_0$ it follows that the entire $\phi$-orbit of $\gamma$ belongs to $\Gamma \cap \Sigma_0$, and then by Lemma~\ref{lem:separating-multicurve} this orbit separates $\Sigma$.  Therefore, at least some components of the orbit must lie on the boundaries of closures of components of $\Sigma \setminus \Gamma$, hence on $\partial \Sigma_0$ in particular.  But the $\phi$-action fixes $\partial \Sigma_0$ setwise as well, so since these components are in the $\phi$-orbit of $\gamma$ we must have $\gamma \subset \partial\Sigma_0$.  This proves that $\Gamma \cap \Sigma_0 \subset \partial\Sigma_0 \setminus \partial\Sigma$, and then by definition every component of $\partial\Sigma_0 \setminus \partial\Sigma$ belongs to $\Gamma$, so in fact $\Gamma \cap \Sigma_0 = \partial\Sigma_0 \setminus \partial\Sigma$.

Supposing now that some component
\[ \gamma_0 \subset \partial\Sigma_0 \setminus \partial \Sigma \subset \Gamma \]
is nonseparating in $\Sigma$, let us denote the curves in its $\phi$-orbit by
\[ \gamma_i = \phi^i(\gamma_0), \qquad i=0,1,\dots,n-1, \]
where each of the $\gamma_i$ are distinct and $n \geq 1$ is the smalllest positive integer satisfying $\phi^n(\gamma_0) = \gamma_0$.  We must actually have $n \geq 2$, because otherwise $\phi(\gamma_0) = \gamma_0$ and then Lemma~\ref{lem:separating-multicurve} would say that $\gamma_0$ must have been separating after all.  Since $\phi$ fixes $\partial\Sigma_0 \setminus \partial\Sigma$ setwise, it follows that each $\gamma_i$ is a component of $\partial\Sigma_0 \setminus \partial\Sigma$ as well.  

Let $\alpha \subset \Sigma$ be a simple closed curve that intersects $\gamma_0$ transversely in a single point $p$.  We orient $\alpha$ so that it travels out of $\Sigma_0$ at $p$; since it eventually returns to $\Sigma_0$ it must meet some other component of $\partial\Sigma_0$ transversely, so we let $q$ be the first point where it does so, with $q$ belonging to some component $\beta \subset \partial\Sigma_0 \setminus \partial\Sigma$.  Then we can replace the portion of $\alpha$ from $q$ to $p$ with a properly embedded arc inside $\Sigma_0$, so that
\[ \alpha \cap \partial \Sigma_0 = \{p\} \sqcup \{q\} \subset \gamma_0 \sqcup \beta. \]
The curve $\beta$ must then lie in the $\phi$-orbit of $\gamma_0$: otherwise $\alpha$ would be a simple closed curve meeting this orbit transversely in a single point, and this would contradict the fact that by Lemma~\ref{lem:separating-multicurve} this orbit must be separating.  Thus we have
\[ \beta = \gamma_i = \phi^i(\gamma_0) \]
for some integer $i$, with $1 \leq i \leq n-1$.

We now consider the prime factorization of $n \geq 2$, writing
\[ n = \prod_{j=1}^k p_j^{e_j} \]
where the $p_j$ are distinct primes and each $e_j$ is a positive integer.  Since $i$ is between $1$ and $n-1$, there is some index $j$ for which $i$ is not a multiple of $p_j^{e_j}$.  Dropping the $j$ subscripts for convenience, so that $p^e$ divides $n$ but does not divide $i$, we will examine the branched cyclic cover
\[ \Sigma_{p^e}(K), \]
in which $K$ lifts to a fibered knot $K_{p^e}$ with monodromy freely isotopic to $\phi^{p^e}$.

According to \cite[\S5]{gordon-aspects} or Proposition~\ref{prop:resultants}, since $p$ is prime we have
\[ H_1(\Sigma_{p^e}(K); \Z/p\Z) = 0, \]
hence in particular $\Sigma_{p^e}(K)$ is a rational homology sphere.  The multicurve $\Gamma \subset \Sigma$ is still a reducing system for the monodromy $\phi^{p^e}$, but the $\phi^{p^e}$-orbit of $\gamma_0 \subset \Gamma$ is now
\[ \{ \gamma_0,\ \gamma_{p^e},\ \gamma_{2\cdot p^e},\ \gamma_{3\cdot p^e},\ \dots,\ \gamma_{n-p^e} \}. \]
The curve $\beta = \gamma_i$ does not belong to this $\phi^{p^e}$-orbit since $i$ is not a multiple of $p^e$, so now the curve $\alpha$ from above meets this orbit transversely in a single point, hence this $\phi^{p^e}$-orbit is non-separating in $\Sigma$.  But since $\Sigma_{p^e}(K)$ is a rational homology sphere, this contradicts Lemma~\ref{lem:separating-multicurve}, and so our original curve $\gamma_0$ must have separated $\Sigma$ after all.
\end{proof}

We can now say quite a bit about the monodromies of fibered satellite knots in $S^3$.  Indeed, let us suppose for the rest of this section that $K\subset S^3$ is a fibered knot whose monodromy $h:\Sigma\to \Sigma$ is freely isotopic to a map $\phi:\Sigma\to\Sigma$ in Nielsen--Thurston form, with nonempty reducing set $\Gamma$. 

By Lemma \ref{lem:outermost-separating}, we can write \begin{equation*}\label{eqn:notation-sat} \partial\Sigma_0 = \partial\Sigma \sqcup (\Gamma \cap \Sigma_0) = \partial\Sigma \sqcup c^1 \sqcup \dots \sqcup c^{k}, \end{equation*}
where each $c^i$ is a disjoint union of  separating curves on $\Sigma$, \[c^i = c^i_0\sqcup \dots \sqcup c^i_{m_i-1}\] with $m_i\geq 1$, satisfying \[\varphi(c^i_j) = c^i_{j+1},\] with subscripts understood mod $m_i$. That is, each $c^i$ represents a single $\varphi$-orbit of  the components of $\partial \Sigma_0\setminus \partial \Sigma$. We can then write \[ \overline{\Sigma \setminus \Sigma_0} =  S^1 \sqcup \dots \sqcup S^{k}, \] where each $S^i$ is a disjoint union of connected subsurfaces of $\Sigma$, \[S^i = S^i_0\sqcup \dots \sqcup S^i_{m_i-1}, \] whose boundaries are given by $\partial S^i_j = c^i_j$ and for which \[\varphi(S^i_j) = S^i_{j+1}.\] See Figure~\ref{fig:nielsen-thurston} for an illustration of this decomposition in an example where $k=2$; note that each $S^i_j$ has genus at least one by the minimality of $\Gamma$.
\begin{figure}
\begin{tikzpicture}
\draw (-2,0) arc (180:360:2 and 0.2) node[right] {$\partial\Sigma$};
\draw[thin, gray!50] (-2,0) arc (180:0:2 and 0.2);
\draw (-2,0) to[out=90,in=270] ++(-1.25,2);
\draw (2,0) to[out=90,in=270] ++(1.25,2);
\foreach \i in {-1.5,0,1.5} {
  \begin{scope}
  \draw (\i,0.85) coordinate (gcenter) ++(0,0.5) ++(225:0.5) arc (225:315:0.5);
  \clip (gcenter) ++(0,0.5) -- ++(225:0.5) arc (225:315:0.5) -- ++(135:0.5);
  \draw (gcenter) ++(0,-0.45) ++(45:0.6) arc (45:135:0.6);
  \end{scope}
}
\foreach \i in {-2.75,-1.5,0.25,1.5,2.75} {
  \begin{scope}
  \path (\i,2) ++(-0.5,0) coordinate (leftend);
  \clip (leftend) ++ (0,-0.3) rectangle ++(1,0.6);
  \draw[ultra thick] (leftend) arc (180:360:0.5 and 0.1);
  \draw[gray!50] (leftend) arc (180:0:0.5 and 0.1);
  \end{scope}
}
\foreach \i in {-2.75,0.25,1.5} \draw (\i,2) ++(0.5,0) to[out=270,in=270,looseness=3] ++(0.25,0);
\draw (-1.5,2) ++(0.5,0) to[out=270,in=270] ++(0.75,0);
\foreach \i in {-2.75,-1.5} {
 \draw (\i,2) ++(-0.5,0) to[out=90,in=90,looseness=5] ++(1,0);
}
\foreach \i in {0.25,1.5,2.75} {
 \draw (\i,2) ++(-0.5,0) to[out=90,in=90,looseness=7] ++(1,0);
}
\foreach \i/\j in {-2.75/2.5,-1.5/2.5, 0.25/2.5, 1.5/2.5, 2.75/2.5, 0.25/3, 1.5/3, 2.75/3} {
  \begin{scope}
  \draw (\i,\j) coordinate (gcenter) ++(0,0.3) ++(225:0.3) arc (225:315:0.3);
  \clip (gcenter) ++(0,0.3) -- ++(225:0.3) arc (225:315:0.3) -- ++(135:0.3);
  \draw (gcenter) ++(0,-0.27) ++(45:0.36) arc (45:135:0.36);
  \end{scope}
}
\coordinate (leftcenter) at (-2.125, 3.5);
\draw[-latex] (leftcenter) ++(0.35,0) to[bend right=30] ++(-0.7,0);
\draw[-latex] (leftcenter) ++(-0.25,-0.2) to[bend right=30] ++(0.5,0);
\coordinate (rightcenter1) at (2.125,4);
\draw[-latex] (2.125,4) ++(0.35,0) to[bend right=25] ++(-1.95,0);
\draw[-latex] (0.875,4) ++(-0.25,-0.2) to[bend right=20] ++(0.5,-0.1);
\draw[-latex] (2.125,4) ++(-0.25,-0.3) to[bend right=20] ++(0.5,0.1);
\draw[decorate,decoration={brace,amplitude=5pt,raise=0.5pt,mirror}] (3.5,0) -- node[midway,right,outer sep=4pt] {$\Sigma_0$} ++(0,2);
\draw[decorate,decoration={brace,amplitude=3pt,raise=0.5pt}] (-3.35,1.8) -- node[midway,left,outer sep=2pt] {$\Gamma$} ++(0,0.4);
\draw (-3.25,0.75) coordinate (arrowcenter) ++ (0:0.35) arc (360:60:0.35);
\draw[-latex] (arrowcenter) ++(60:0.35) -- ++(-15:0.05);
\node[below left, inner sep = 0.25cm] at (arrowcenter) {$\phi_0$};
\draw[decorate,decoration={brace,amplitude=5pt,raise=0.5pt}] (-3.25,3.75) -- node[midway,above,outer sep=5pt] {$S^1$} ++(2.25,0);
\draw[decorate,decoration={brace,amplitude=5pt,raise=0.5pt}] (-0.25,4.35) -- node[midway,above,outer sep=5pt] {$S^2$} ++(3.5,0);
\end{tikzpicture}
\caption{An example of a Nielsen--Thurston decomposition.}\label{fig:nielsen-thurston}
\end{figure}
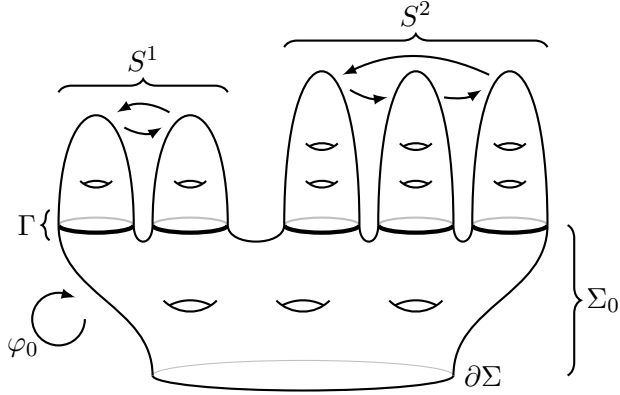

Now, the suspension of each $c^i$ is a torus $T^i$ in the mapping torus $M_\varphi \cong S^3\setminus \nu(K)$. This torus necessarily bounds a solid torus in $S^3$ and thus realizes $K$ as a satellite knot: the suspension of $S^i$ gives the complement $S^3 \setminus \nu(C^i)$ of the associated companion knot $C^i \subset S^3$, and the rest of the mapping torus corresponds to the complement of the pattern $P^i
$ in the solid torus. This means that the companion  knot $C^i$ is fibered, with monodromy which is freely isotopic to the map
\[ \varphi^{m_i}|_{S^i_0} : S^i_0 \to S^i_0. \] Note that this restriction map is in Nielsen--Thurston form, with reducing set given by $\Gamma\cap \textrm{int} \,S^i_0$. 

We will use the notation introduced above throughout this paper. One quick corollary, which we will not use, but record here in case it is helpful to others, is that every component of the reducing set of a fibered knot in $S^3$ is separating:

\begin{lemma} \label{lem:reducing-separating}
Let $K \subset S^3$ be a fibered knot whose monodromy is freely isotopic to a map $\phi: \Sigma \to \Sigma$ in Nielsen--Thurston form, with reducing system $\Gamma$.  Then each component of $\Gamma$ separates $\Sigma$.
\end{lemma}

\begin{proof}
We will induct on the number $n \geq 1$ of components of $\Sigma \setminus \Gamma$.  We know by Lemma~\ref{lem:separating-multicurve} that $\Gamma$ separates $\Sigma$ if it is nonempty.  If $n=1$ then $\Sigma \setminus \Gamma$ is connected, so $\Gamma$ is empty and the proposition is vacuously true.  We will thus assume that $n \geq 2$, and adopt the notation introduced above.

A component  $\gamma\subset\Gamma$ is either $c^i_j$ for some $i,j$, or lies in the interior of some $S^i_j$.  We will assume the latter, since otherwise we already know that $\gamma$ is separating, and we may suppose without loss of generality that $\gamma\subset \textrm{int} \,S^i_0$. 
Then $\gamma$ is a component of $\Gamma\cap \textrm{int} \,S^i_0$, which is the reducing set for the monodromy of the corresponding companion knot $C^i\subset S^3$, as noted above. Now, the complement
\[ S^i_0 \setminus (\Gamma \cap S^i_0) \]
has strictly fewer than $n$ components, because these are a subset of the $n$ components of $\Sigma \setminus \Gamma$ but  do not include the component containing $\partial \Sigma$.  Thus by hypothesis $\gamma$ is separating in $S^i_0$, and therefore separates $\Sigma$ as well.  This completes the proof by induction.
\end{proof}

More importantly, we may now describe the meridians and Seifert longitudes of the companion knots $C^i$ on the corresponding satellite tori $T^i$ purely in terms of the restriction of $\varphi$ to the outermost component $\Sigma_0$ and the meridian of $K$. 
Let $\varphi_0$ denote the restriction of $\varphi$ to $\Sigma_0$. By definition, the mapping torus $M_{\varphi_0}$ has boundary \[\partial M_{\varphi_0}  = T \sqcup T^1 \sqcup \dots \sqcup T^k,\] where $T$ is the suspension of $\partial \Sigma$ in this mapping torus, and each $T^i$ is the suspension of \[c^i=c^i_0\sqcup\dots\sqcup c^i_{m_i-1} \subset \partial \Sigma_0\setminus \partial \Sigma.\] The longitude $\lambda_K$ is  the image of $\partial \Sigma$ in $T$, and the meridian $\mu_K$ is a curve on $T$ dual to $\lambda_K$.

\begin{lemma}\label{lem:sat-longitude}
The Seifert longitude $\lambda_{C^i}$ of the companion knot $C^i$ is given by the image of  $c^i_j$ in $T^i$ for any $j=0,\dots,m_i-1$. The pattern $P^i$ has winding number $m_i$.
\end{lemma}

\begin{proof}
Orient  $c^i_j$ as the boundary of $S^i_j$, and observe  that the images of  $c^i_0, \dots, c^i_{m_i-1}$ are oriented isotopic primitive curves in $T^i$. Since each $S^i_j$ is a fiber of the fibered knot $C^i$, the Seifert longitude $\lambda_{C^i}$ is given by the image of  $c^i_j$ in $T^i$ for any $j=0,\dots,m_i-1$. Moreover,  the fiber $\Sigma$ in the mapping torus  \[M_\varphi \cong S^3 \setminus \nu(K)\] intersects $T^i$ transversely in the oriented multicurve  $c^i_0\sqcup\dots\sqcup c^i_{m_i-1}$. In particular, \[[\Sigma\cap T^i] = m_i[c^i_j]\in H_1(T^i;\mathbb{Z}), \] for each $j$. It then follows from  Lemma \ref{lem:winding-number-satellite} that $P^i$ has winding number $m_i$.
\end{proof}

\begin{proposition}\label{prop:sat-meridian}
The meridian $\mu_{C^i}$ of the companion knot $C^i$ is the unique primitive isotopy class of curves in $T^i$ satisfying  \[[\mu_{C^i}] = m_i[\mu_K]\] as elements of $H_1(M_{\varphi_0};\mathbb{Z})$.
\end{proposition}

\begin{proof} 
It suffices without loss of generality to prove this in the case that $i=1$. Note that \[S^3\setminus \nu(K)=M_{\varphi_0}\cup E_{C^1}\cup E_{C^2}\cup \dots \cup E_{C^k} \] where each knot complement $E_{C^j}=S^3\setminus \nu(C^j)$ is glued onto $M_{\varphi_0}$ along the torus $T^j$. In particular, if we omit the complement of $C^1$ from this decomposition, then the resulting manifold \[E_{P^1}=M_{\varphi_0} \cup E_{C^2}\cup \dots \cup E_{C^k}\] is the complement of the pattern $P^1$ in the corresponding solid torus $V^1\subset S^3$, with boundary \[\partial E_{P^1} = T \sqcup T^1.\] By Definition \ref{def:generalized-companion}, the meridian $\mu_{C^1}\subset T^1$ of the companion knot $C^1$ is given by the longitude $\alpha_{V^1}$ of the solid torus \[V^1\cong S^1\times D^2.\] Note that $\alpha_{V^1}$ is simply the boundary $\partial D$ of the meridional disk \[D = \{\pt\}\times D^2.\] It therefore suffices to show that \begin{equation}\label{eq:relation-D}[\partial D] = m_1[\mu_K]\in H_1(M_{\varphi_0};\Z),\end{equation} and that $\partial D$ is the unique primitive isotopy class of curves in $T^1$ satisfying this relation.

Let us prove the uniqueness first. Suppose for a contradiction that $\gamma_1$ and $\gamma_2$ are two distinct primitive isotopy classes of curves in $T^1$ satisfying \[[\gamma_1]=[\gamma_2] = m_1[\mu_K]\in H_1(M_{\varphi_0};\Z).\] Since the solid torus $V^1$ is obtained from $M_{\varphi_0}$ by attaching the complements $E_{C^j}$ for $j=2,\dots,k$, and then Dehn filling $T$ along $\mu_K$, we therefore have that  \[[\gamma_1]=[\gamma_2] = 0 \in H_1(V^1;\Z).\] But then half lives, half dies implies that the primitive curves $\gamma_1$ and $\gamma_2$ are homologous and hence isotopic in $T^1 = \partial V^1$, a contradiction.

It remains to prove the relation in \eqref{eq:relation-D}. As discussed in \S\ref{sec:generalized-satellite}, we know that this relation holds in $H_1(E_{P^1};\Z)$. Indeed, let us assume that $D$ is transverse to $P^1\subset V^1$, and let \[D' = D\cap E_{P^1}.\] Then $D'$ is a planar surface in $E_{P^1}$ whose boundary is the union of $\partial D\subset T^1$ with a multicurve in $T$ homologous to $m_i$ copies of $\mu_K$. To prove that the relation remains true after removing  $E_{C^2},\dots,E_{C^k}$ from $E_{P^1}$ to form $M_{\varphi_0}$ it suffices to show that $D'$ intersects each of the tori $T^2,\dots,T^k\subset E_{P^1}$ in a nullhomologous curve (after perturbing $D'$ so that it is transverse to these tori). Without loss of generality, it is enough to show that \[[D'\cap T^2] = 0\in H_1(T^2;\Z).\] Note that \[[D'\cap T^2] = 0 \in H_1(E_{C^2};\Z)\] since the intersection $D'\cap T^2$ bounds the planar surface \[D''=D'\cap E_{C^2} \subset E_{C^2}.\] As the kernel of the inclusion-induced map \[H_1(T^2;\Z)\to H_1(E_{C^2};\Z)\] is generated by the class of the Seifert longitude $\lambda_{C^2}$, it follows that \[[D'\cap T^2] = [\partial D'']=k[\lambda_{C^2}]\in H_1(T^2;\Z)\] for some integer $k$, which we want to show is zero. 

After a standard innermost circle argument to remove nullhomotopic components of the intersection $D'\cap T^2$, and then compressing $D''$ along any compressing disks that might exist, we can therefore write \[D'' = D''_1\sqcup \dots \sqcup D''_m\] where each component $D''_i$ is a connected, oriented, incompressible planar surface in $E_{C^2}$ whose boundary is a union of curves on $T^2$, each isotopic as an unoriented curve to $\lambda_{C^2}$. Since  each $D''_i$ is planar, it cannot be isotopic to a fiber of the positive-genus fibered knot $C^2$. Then \cite[Proposition 4.7]{bs-apoly} says that $D''_i$ is separating. In particular, $D''_i$ together with some portion of $T^2$ bounds a component of $E_{C^2}\setminus D''_i$ and therefore represents the trivial class \[[D''_i]=0\in H_2(E_{C^2},T^2).\] This class is  sent to \[[\partial D''_i] = 0 \in H_1(T^2;\Z)\] under the map in the long exact sequence associated with the pair. It follows that \[[D'\cap T^2] = [\partial D''] = [\partial D''_1] + \dots + [\partial D''_m] = 0\in H_1(T^2;\Z),\] as desired, completing the proof of this proposition.
\end{proof}

This proposition then leads to the following, which is a key input in our proof of Theorem \ref{thm:hfk-j-finite} and subsequently Theorem \ref{thm:main} in the last section.

\begin{theorem}\label{thm:determine-C}
Suppose that $\varphi_0$ is freely isotopic to a pseudo-Anosov map $\psi_0:\Sigma_0\to\Sigma_0$. Then the knots $C^1,\dots,C^k$ together with the conjugacy class of $\psi_0$ determine the knot $K\subset S^3$ up to finitely many possibilities.
\end{theorem}

\begin{proof}
We will show that the knots $C^1,\dots,C^k$ and the conjugacy class of $\psi_0$ determine the complement of $K\subset S^3$, and hence the knot itself \cite{gordon-luecke-complement}, up to finitely many possibilities. 

Since the longitude $\lambda_K$ is given by the image of $\partial \Sigma$ in the torus $T\subset \partial M_{\varphi_0}$, and the meridian $\mu_K$ is dual to this longitude, there is a homeomorphism $h_0:\Sigma_0\to\Sigma_0$ such that:
\begin{itemize}
\item $h_0$ restricts to the identity on $\partial \Sigma$,
\item $h_0$ is freely isotopic to $\psi_0$ by a free isotopy which restricts to the identity on $\partial \Sigma_0\setminus \partial \Sigma$, and
\item $\mu_K$ is the suspension of a point $p\in \partial \Sigma$ in the mapping torus $M_{h_0}\cong M_{\varphi_0}$.
\end{itemize}
Given such an $h_0$, the meridian $\mu_{C^i}$ on the torus $T^i\subset M_{h_0}$ is completely determined according to Proposition \ref{prop:sat-meridian}, and the longitude $\lambda_{C^i}$ is the image of $c^i_0$ in $T^i$. That is, given $h_0$ as above, there is a canonical way of gluing the knot complements $E_{C^1},\dots, E_{C^k}$ to $M_{h_0}$ to form the complement $S^3\setminus \nu(K)$, and the resulting complement depends only on the isotopy class of $h_0$ rel $\partial \Sigma$. So the only ambiguity in the possibilities for $K$, having fixed the knots $C^1,\dots,C^k$ and the map $\psi_0$, is the ambiguity in the possibilities for the isotopy class of $h_0$ rel $\partial \Sigma_0$.

Let $h_0$ be as above. The \emph{fractional Dehn twist coefficient} of $h_0$ at $\partial \Sigma$ is a quantity \[c(h_0)\in \Q,\] introduced by Honda--Kazez--Mati{\'c} \cite{hkm-veering} as a reformulation of Gabai's notion of degeneracy slope, which measures the twisting near $\partial \Sigma$ in the free isotopy between $h_0$ and $\psi_0$. In this case, $c(h_0)$ is precisely the fractional Dehn twist coefficient of the monodromy of the knot $K$. Since $K$ is a knot in $S^3$, we have, for instance by \cite[Theorem 1]{hedden-mark-fdtc}, that \begin{equation}\label{eq:fdtc}|c(h_0)|\leq 1.\end{equation} Now suppose that $h_0$ and $h_0'$ are two maps satisfying the three conditions above. Then, just from the first two conditions, there is some integer $n$ such that $h_0'$ is isotopic to $t_\partial^n \circ h_0$ rel $\partial \Sigma_0$, where $t_\partial$ is a Dehn twist about a curve in $\Sigma_0$ parallel to the boundary component $\partial \Sigma$. It follows that \[c(h_0')-c(h_0) = n.\] But the constraint in \eqref{eq:fdtc}  says that \[c(h’_0) = c(h_0) + n \in [-1,1],\] leaving three possibilities for $n$.  In particular, there are at most three possibilities for the isotopy class of $h_0$ rel $\partial \Sigma_0$, and hence at most three possibilities for $K$. 

Note that conjugating $\psi_0$ corresponds to conjugating the possible candidates for $h_0$. But conjugating a given $h_0$ preserves  $M_{h_0}$ together with the relevant peripheral structures (the meridians and longitudes of $C^1,\dots,C^k$) up to homeomorphism, and thus does not change the possibilities for $S^3\setminus \nu(K)$. So, in fact, there are only finitely many possibilities for $K$ once we specify the knots $C^1,\dots, C^k$ and the mere \emph{conjugacy class} of $\psi_0$.
\end{proof}

\section{Immersed curves and a satellite inequality} \label{sec:satellite-inequality}
Our goal in this section is to prove Theorem \ref{thm:gensat}. The main tool used in our proof is an immersed curves formulation of bordered Heegaard Floer homology for 3-manifolds with torus boundary and knots therein, as developed in \cite{hrw,hrw-properties}. We provide a very cursory review of these topics, tailored to our needs; for more details, see the references above as well as \cite{LOT, LOT-morphism}.

\subsection{Bordered Heegaard Floer homology}
Let us fix a torus $T$ and a point $\bullet$ therein. Let \[T_\bullet = T\setminus \{\bullet\}.\] We also fix a pair of simple oriented closed curves $\eta,\xi\subset T_\bullet$ which intersect in a single point, with \[\eta\cdot\xi = +1.\]A \emph{bordered manifold with torus boundary} is a compact oriented 3-manifold $M$ with torus boundary, together with an orientation-reversing homeomorphism \[\varphi:T\to \partial M,\] considered up to isotopy rel $\bullet$. We will often suppress the parametrization $\varphi$ from the notation.

In \cite{LOT}, Lipshitz, Ozsv{\'a}th, and Thurston define the \emph{torus algebra} $\Alg$. And given a pointed bordered Heegaard diagram $(\mathcal{H},z)$ for a bordered manifold $M$ with torus boundary, they define a certain left differential graded module over $\Alg$ called a \emph{type D structure}, denoted by \[\CFD(\mathcal{H},z).\] Up to homotopy equivalence, this module is an invariant of $M$, and hence also denoted by $\CFD(M)$.

If $M_1$ and $M_2$ are bordered manifolds with torus boundary, then their boundary parametrizations give a canonical way of gluing them to form a closed oriented 3-manifold $-M_1\cup M_2$. The pairing theorem of \cite{LOT}, reinterpreted in \cite{LOT-morphism}, says that the Heegaard Floer homology of this closed manifold is isomorphic to the homology of the space of morphisms between the corresponding type D structures, \[\HFhat(-M_1\cup M_2) \cong H_*\left(\Mor_\Alg\big(\CFD(M_1), \CFD(M_2)\big)\right).\]

There is  a version of these results for knots in bordered manifolds as well.  Let $P\subset M$ be a knot in a bordered manifold with torus boundary.  We can represent this knot by adding a second basepoint $w$ to some pointed bordered Heegaard diagram $(\mathcal{H},z)$ for $M$. This gives rise to a type D structure $\CFD(\mathcal{H}, z, w)$ over $\Alg$, defined in the same way as $\CFD(\mathcal{H}, z)$ except that the differential no longer counts holomorphic disks covering the new basepoint $w$. 

\begin{remark}\label{rmk:dependence-on-arc}
The homotopy equivalence type of $\CFD(\mathcal{H}, z, w)$ is not quite an invariant of $P\subset M$: it also depends on a choice of arc connecting $P$ to the point $\varphi(\bullet)\in\partial M$.  However, our interest is in the pairing formula \eqref{eq:LOT-satellite-pairing} below, which gives the same end result for any choice of arc. We thus abuse notation slightly and write $\CFD(M,P)$ for the homotopy equivalence type of $\CFD(\mathcal{H}, z, w)$.
\end{remark}

Suppose that $M_1$ and $M_2$ are bordered manifolds with torus boundary, and $P\subset M_2$ is a knot. Let $K$ denote the image of $P$ in $-M_1\cup M_2$.  By \cite[Theorem 11.19]{LOT}, reinterpreted in \cite{LOT-morphism}, there is a similar pairing formula for knot Floer homology,
\begin{equation}\label{eq:LOT-satellite-pairing}
\HFK(-M_1 \cup M_2, K) \cong H_*\left(\Mor_\Alg\big(\CFD(M_1), \CFD(M_2,P)\big)\right).
\end{equation}
We will be particularly interested in the case that $P\subset M_2$ is the pattern of a generalized satellite knot $K\subset -M_1\cup M_2$ with satellite torus $\partial M_1 = \partial M_2$.

\subsection{An immersed curves formulation}

We will use a reformulation of type D structures and the pairing formulas above in terms of immersed curves, following the work of Hanselman, Rasmussen, and Watson in \cite{hrw,hrw-properties}. 

By \cite[Theorem 1.5]{hrw}, a type D structure over $\Alg$ determines a collection of compact oriented immersed curves in $T_\bullet$, as long as  the module  satisfies an algebraic constraint called \emph{extendability}, which is known to hold for $\CFD$ of bordered manifolds with torus boundary \cite[Theorem 1.4]{hrw}. The homotopy class of this multicurve is an invariant of the corresponding type D structure up to homotopy equivalence. Moreover, this curve fully encodes the homotopy equivalence type of the corresponding module if one includes some additional decorations, in the form of local systems and gradings on the multicurve. For our purposes, we will be able to ignore these decorations and work only with the underlying immersed multicurves, as explained in a bit.

Suppose that $M$ is a bordered manifold with torus boundary, and consider the immersed curve in $T_\bullet$ determined by $\CFD(M)$. Its image under the boundary parametrization $\phi$ is a multicurve in \[\partial M \setminus \{\phi(\bullet)\},\] which is shown in \cite{hrw} to be independent of $\phi$. This curve is thus an invariant of the underlying manifold $M$, and will be denoted by \[\Gamma(M)\subset \partial M.\]  When equipped with the additional decorations above, this curve is denoted by $\HFhat(M)$ in \cite{hrw}. 

\begin{remark}
We will use the conventions for the local system decorations  in \cite{Hanselman:CFK}, which differ slightly from those in \cite{hrw}. The difference is that in \cite{Hanselman:CFK} the multiplicity of the local system is built into the immersed curve rather than being part of the decoration. This means we do not need to count intersection points with multiplicity when taking Floer homology as in \cite{hrw}, and the multiplicity information is not lost when we discard the local system decoration.
\end{remark}

Although the bordered Floer invariants for manifolds with torus boundary are always extendable, the type D structure $\CFD(M,P)$ corresponding to a knot $P$ in such a manifold $M$ is in general not extendable. Fortunately the extendability hypothesis in the immersed curve structure theorem for type D structures over $\Alg$ can be removed, as  done in \cite[Theorems 5.14 and 5.27]{KWZ}. There is little change to the statement or the proof, except that the resulting immersed multicurves are no longer necessarily compact: they may contain both immersed circles and immersed arcs approaching the puncture at their ends. In fact, a type D structure is extendable if and only if its corresponding immersed multicurve is compact.

Suppose that $P\subset M$ is a knot in a bordered  manifold with torus boundary. Just as before, the image under the boundary parametrization $\varphi$ of the immersed multicurve in $T_\bullet$ determined by  $\CFD(M,P)$ is independent of $\varphi$, and we denote this image by \[\Gamma(M, P)\subset \partial M\setminus \{\varphi(\bullet)\}.\] Technically, this multicurve also depends on a choice of arc connecting $P$ to $\varphi(\bullet)$ since $\CFD(M,P)$ does, as noted in Remark \ref{rmk:dependence-on-arc}.  As before, we may choose an arbitrary arc for our purposes, allowing us to abuse notation slightly.  As a matter of convention, we will often view $\Gamma(M)$ and $\Gamma(M,P)$ as curves in the punctured torus $T_{\bullet}$ via the identification of $T$ with $\partial M$ given by $\varphi$.

The key utility of the immersed curves perspective lies in a reformulation of the pairing theorems. Given two type D structures over $\Alg$, the space of morphisms from one to the other is homotopy equivalent to a certain Lagrangian Floer complex of the corresponding decorated multicurves in $T_\bullet$. This was proved in \cite{hrw} when both modules are extendable, and the argument can be adapted to the non-extendable case as in \cite[Theorem 5.22]{KWZ}. 

\begin{remark}
The authors in \cite{hrw} really proved that for two extendable type D modules $N_1,N_2$ over $\Alg$, the Floer complex of the corresponding decorated multicurves is homotopy equivalent to the box tensor product of the right $A_\infty$-module dual to $N_1$ with $N_2$, as defined in \cite{LOT}. But it is shown in \cite{LOT-morphism} that this tensor product is homotopy equivalent to $\Mor_\Alg(N_1,N_2)$. The pairing theorem in \cite{KWZ} that we appeal to for the non-extendable case is stated and proved using morphsims of type D structures rather than using dual $A_\infty$-modules.
\end{remark}

While the Floer homology of decorated multicurves depends in general on the local systems, we will be able to ignore these decorations in our setting, as explained below.

The Floer homology of a pair of decorated curves on a surface is often easy to compute when the curves are arranged nicely: if multicurves $\Gamma_1$ and $\Gamma_2$ are homotoped to intersect minimally, then the Floer differential vanishes. As a result,  $\dim \HF(\Gamma_1, \Gamma_2)$ is usually just the minimal intersection number of $\Gamma_1$ and $\Gamma_2$, which we denote $i(\Gamma_1, \Gamma_2)$. The only caveat is that admissibility constraints for Floer homology sometimes require that $\Gamma_1$ and $\Gamma_2$ not be put in minimal position. This happens when a component of $\Gamma_1$ is homotopic to a component of $\Gamma_2$, or more generally if a component of $\Gamma_1$ is \emph{commensurable} to a component of $\Gamma_2$, meaning that both curves are homotopic to multiples of the same primitive curve. In this case, admissibility can be achieved by perturbing curves into a slightly non-minimal position, and we have that
\[ \dim \HF(\Gamma_1, \Gamma_2) = i(\Gamma_1, \Gamma_2) + \text{[a correction term]}, \]
where the correction term depends on local systems but is always nonnegative (see \cite[Corollary 4.10]{hrw}). Thus we always have
\[ \dim \HF(\Gamma_1, \Gamma_2) \ge i(\Gamma_1, \Gamma_2), \]
with equality holding when no components of $\Gamma_1$ and $\Gamma_2$ are commensurable. This relationship will be enough for our purposes in the proof of Theorem \ref{thm:gensat}. Note that $i(\Gamma_1, \Gamma_2)$ does not depend on the decorations on the curves, which is why we may safely ignore them as mentioned above.

This discussion together with the pairing formula in \eqref{eq:LOT-satellite-pairing} then yields the following:

\begin{proposition}\label{prop:curves-pairing}
Suppose  that $M_1$ and $M_2$ are  bordered manifolds with torus boundary, and  $P\subset M_2$ is a knot. Let $K$ denote the image of $P$ in $-M_1\cup M_2$. Then
\[ \dim \HFK(-M_1\cup M_2, K) = \dim \HF( \Gamma(M_1), \Gamma(M_2, P) ) \ge i( \Gamma(M_1), \Gamma(M_2, P) ), \]
and equality holds if $\Gamma(M_1)$ and $\Gamma(M_2, P)$ have no commensurable components.
\end{proposition}

\begin{remark}
\label{rmk:commen}
In this proposition, we are viewing $\Gamma(M_1)$ and $\Gamma(M_2, P)$ as curves in $T_\bullet$ via the boundary identifications $T\to\partial{M_i}$; this will be standard practice in what follows.
\end{remark}

Given the discussion above, it will be important to choose representatives of homotopy classes of immersed multicurves in $T_\bullet$ that intersect minimally. One way to do this is to \emph{pull tight} both curves, as described in \cite[Section 7.1]{hrw}. Informally, we imagine that each time a curve comes near the puncture $\bullet$, not counting the ends of immersed arcs, it wraps around a circular ``peg" $B_\epsilon(\bullet)$ of some radius $\epsilon$ and has minimal length subject to that constraint. Some portions of the curve lie on the circle $\partial B_\epsilon(\bullet)$---these are called \emph{corners} of the pulled tight curve---and the remaining portions of the curve are straight line segments, as interpreted with respect to the standard flat metric on the torus coming from its universal cover $\R^2$. The curve never enters the interior of $B_\epsilon(\bullet)$ except near the endpoints of immersed arcs, which approach $\bullet$ along straight lines. 

To ensure that two pulled tight multicurves intersect transversely, we pull them tight with respect to different peg radii. In fact, it is helpful to use a different peg radius for each corner of each curve. If these radii are chosen with a little care (in particular, a corner at which a curve makes a sharper turn has a smaller peg radius), then pulling both multicurves tight ensures they intersect minimally. A curve in such an arrangement, which we often draw in the covering space $\R^2$, is called a \emph{pegboard representative}; again, see \cite[\S7]{hrw} for details.

We sometimes consider a singular pegboard representative obtained by taking all peg radii to zero, in which the curve becomes piecewise linear with the endpoints of line segments at the puncture. This allows us to define a special type of corner, which we will paradoxically call a \emph{straight corner} (the term $\pi$-corner is used in \cite{hrw}). A straight corner in a nonsingular pegboard representative is a corner at which the curve does not change direction in the corresponding singular representative; this occurs when the center of the peg for the given corner is colinear in the plane with the centers of the pegs for  the preceding and following corners along the curve. Straight corners will require special attention in arguments we will make about what happens to the intersection number of two curves when  pulling one of them through a peg.

\subsection{Restrictions on immersed curves}

Here, we collect and prove some facts about the multicurves in $T_\bullet$ determined by type D structures over $\Alg$, which we will use in our proof of Theorem \ref{thm:gensat}. We begin with a little bit more background on how these modules and their corresponding multicurves are related; see \cite[\S2]{hrw} for further details.

Recall that the torus algebra $\Alg$ has two distinguished idempotents, $\iota_0$ and $\iota_1$, satisfying $1 = \iota_0+\iota_1$. Let $\mathcal{I}$ denote the subring of idempotents. Then a type D structure over $\Alg$ consists of a vector space $V$ over $\F$ which splits as a direct sum over a left action of these idempotents, \[V\cong \iota_0V \oplus \iota_1 V,\] together with a structure map \[\delta^1:V\to \Alg\otimes_\mathcal{I} V\] satisfying certain relations. Recall that we fixed dual oriented curves $\eta,\xi\subset T_\bullet$ at the beginning. If $\Gamma$ is the multicurve on $T_\bullet$ corresponding to $V$, then the intersection points in $\Gamma\cap \eta$ correspond to generators of $\iota_0V$ and the intersection points in $\Gamma\cap \xi$ correspond to generators of $\iota_1V$. Moreover, the mod 2 gradings of these generators are given by the signs of the intersection points. Therefore, the algebraic intersections of $\Gamma$ with $\nu$ and $\xi$ encode the Euler characteristics of the two idempotent summands, which are invariants of the homotopy equivalence type of $V$,
\begin{align*}
\Gamma \cdot \eta &= \chi(\iota_0V),\\
\Gamma \cdot \xi &= \chi(\iota_1V).
\end{align*}
On the other hand, these algebraic intersections determine and are determined by the homology class of $\Gamma$ in $H_1(T)\cong \Z^2$. The last general thing we will say is that a path in $\Gamma$ from one intersection point of $\Gamma \cap(\eta\cup \xi)$ to a consecutive intersection point corresponds to a nonzero component of the structure map $\delta^1$ relating the  corresponding generators of $V$.

\begin{proposition}\label{prop:curve-homology-class}
Let  $P\subset M$ be a knot in a bordered manifold  with torus boundary. Then $\Gamma(M, P)$ is homologous in $\partial M$ to some (possibly zero) multiple of the rational longitude of $M$, and the same is true of each compact component of $\Gamma(M, P)$. Moreover, if $b_1(M) = 1$ then the homology class of $\Gamma(M,P)$ is nonzero.
\end{proposition}

\begin{proof}
This follows from \cite[Corollary 6.6]{hrw}, which makes the corresponding claim for $\Gamma(M)$. In particular, that corollary says that  $\Gamma(M)$ is homologous to a multiple of the rational longitude, and if $b(M)=1$ then it is nonzero in homology. The same proof shows that each compact component of $\Gamma(M, P)$ is homologous to a multiple of the rational longitude. For the claims about the  homology class of $\Gamma(M,P)$, it  suffices to show that $\Gamma(M,P)$ is homologous to $\Gamma(M)$. 

For this, let $(\mathcal{H},z,w)$ be a doubly-pointed bordered Heegaard diagram for $P\subset M$. Note that $\CFD(\mathcal{H},z)$ and $\CFD(\mathcal{H},z,w)$ have the same graded generators, and hence the same Euler characteristics, in each idempotent summand. Since these Euler characteristics determine the homology classes of $\Gamma(M)$ and $\Gamma(M,P)$, as described above, these multicurves are homologous, as desired.
\end{proof}

The next proposition together with Proposition \ref{prop:curve-homology-class} will be  key for our proof of Theorem \ref{thm:gensat}:

\begin{proposition}\label{prop:noncompact}
Let $M$ be a bordered manifold with torus boundary satisfying $H_2(M)=0$, and suppose that $P\subset M$ is a knot with nonzero winding number. Then every compact component of $\Gamma(M,P)$ is nullhomologous, and therefore $\Gamma(M, P)$ contains at least one noncompact component.
\end{proposition}

Our proof of Proposition \ref{prop:noncompact} makes use of the gradings on bordered Floer invariants. Let $P\subset M$ be a knot in a bordered manifold with torus boundary, where $H_2(M)=0$. Then
\begin{equation}\label{eq:h2}
H_2(M,\partial M)\cong \Z,
\end{equation}
as shown in \S\ref{sec:generalized-satellite}.  The full grading defined on the bordered invariant $\CFD(M,P)$ takes values in a noncommutative group, but the noncommutativity arises from a portion of the grading, the Maslov component, that we will ignore. For our purposes, $\CFD(M,P)$ has a grading valued in the quotient \[\frac{\tfrac 1 2 \Z \times \tfrac 1 2 \Z \times \Z}{\langle\Pi_0\rangle},\] where we think of the first $\tfrac 1 2 \Z \times \tfrac 1 2 \Z$ as the half-integer lattice points in \[H_1(T)\cong\Z^2\cong \Z\langle \eta\rangle \oplus \Z\langle \xi\rangle,\] and \[\Pi_0 = (a_0, b_0, c_0)\] is an element corresponding to a generator of \eqref{eq:h2}.

The first two components of the grading are called the \emph{spin$^c$ component}, and the third grading component is called the \emph{Alexander component}. Although we will not need this, the pair \[(a_0,b_0)\in\Z\langle \eta\rangle \oplus \Z\langle \xi\rangle\] represents the image of a generator under the map \[H_2(M,\partial M) \to H_1(\partial M),\] where we identify the latter with $H_1(T)$ via the boundary parametrization $T \to \partial M$. This fact about the spin$^c$ component of $\Pi_0$ is well-known; see, for instance, the discussion in \cite[\S2.2]{hrw-properties}.

We will be more interested in the Alexander component $c_0$ of $\Pi_0$, which is understood by experts but less well-known. Let us first recall roughly how it is defined. Let $(\mathcal{H},z,w)$ be a doubly-pointed bordered Heegaard diagram for $P\subset M$. Let  $x$ and $y$ be two generators in the same idempotent summand of the type D structure $\CFD(\mathcal{H},z,w)$, and let $B$ be the domain of a Whitney disk in $\pi_2(x,y)$. Then the difference between the Alexander gradings of $x$ and $y$ is given by \[n_z(B)-n_w(B).\] Note that $B$ is well-defined up to adding periodic domains.  The set of periodic domains is given by \[\pi_2(x,x)\cong H_2(M,\partial M) \cong \Z,\] and is thus generated by some domain $B_0$. Then $\Pi_0$ is the grading of $B_0$, and is well-defined up to sign. In particular, we have that \[c_0 = n_z(B_0)-n_w(B_0),\] up to sign. The following lemma characterizes $c_0$ in terms of the winding number of $P$:

\begin{lemma}\label{lem:periodic-grading}
Let $M$ be a bordered manifold with torus boundary satisfying $H_2(M)=0$, and suppose that $P\subset M$ is a knot with winding number $w_P$. Let $\Pi_0 = (a_0, b_0,c_0)$ be the tuple defined above, where $c_0\geq 0$. Then $c_0 = n\cdot w_P$, where $n \geq 1$ is the order of the rational longitude of $M$.
\end{lemma}

\begin{proof}
Let $B_0$ be as above. Then $B_0$ determines a 2-chain in the Heegaard surface $\Sigma$ whose boundary consists of a collection of closed $\alpha$ and $\beta$ curves along with a 1-chain in the union of the $\alpha$-arcs and $\partial \Sigma$. Attaching copies of the $\alpha$ and $\beta$ attaching disks gives a 2-chain $S$ in $M$ with boundary in $\partial M$, which represents the generator of $H_2(M,\partial M)$. Orient $P$ so that it has nonnegative algebraic intersection with $S$. Then on one hand, we have that \[n\cdot w_P = [P]\cdot [S]\] by Definition \ref{def:winding-number}. On the other, we can calculate this algebraic intersection number via the Heegaard diagram: $P$ is contained entirely in the $\alpha$ and $\beta$ handlebodies away from the attaching disks, except where it intersects $\Sigma$ at $z$ and $w$, with positive and negative signs, respectively, so we have \[[P]\cdot [S] = n_{z}(B_0) - n_w(B_0) = c_0.\] Combining these two equations proves the lemma.
\end{proof}

We now have the ingredients needed to prove Proposition \ref{prop:noncompact}:

\begin{proof}[Proof of Proposition \ref{prop:noncompact}]
Since $H_2(M)=0$ implies that \[H^1(M)\cong H_2(M,\partial M) \cong \Z,\] and hence that $b_1(M)=1$, Proposition \ref{prop:curve-homology-class} implies that the homology class of $\Gamma(M,P)$ is nontrivial. The claim that $\Gamma(M,P)$ has at least one noncompact component will therefore follow from the claim that each compact component is nullhomologous. So we need only prove the latter.

Let $\gamma$ be a compact component of $\Gamma(M, P)$ and let \[\tilde x\in \Gamma(M,P)\cap (\eta\cup \xi)\] be a point on that component corresponding to some generator $x$ of the type D structure associated with some doubly-pointed bordered Heegaard diagram for $P\subset M$. We can think of $\gamma$ as a loop from $\tilde x$ to itself and use it to compute the grading difference \[\gr(x) - \gr(x) \in \tfrac 1 2 \Z \times \tfrac 1 2 \Z \times \Z,\] which we know is trivial in the quotient by $\langle \Pi_0 \rangle$. Let \[\tilde x = \tilde x_0, \tilde x_1,\dots,\tilde x_m = \tilde x\] be the intersection points of $\Gamma\cap (\eta \cup \xi)$ in order as one traverses $\gamma$, and let $x_i$ be the generator corresponding to $\tilde x_i$. As mentioned at the beginning of this subsection, the existence of the segment from $\tilde x_i$ to $\tilde x_{i+1}$  means that there is a component of the structure map from $x_i$ to $x_{i+1}$ (tensored with an algebra element). But this structure map counts holomorphic curves corresponding to domains which avoid the $z$ and $w$ basepoints of the Heegaard diagram. It follows from the discussion above that the Alexander component of the grading difference between $x_i$ and $x_{i+1}$ is 0. The Alexander grading difference between $x$ and $x$, as measured by the loop $\gamma$, is therefore also $0$.

On the other hand, the spin$^c$ component of this grading difference captures the homology class of the loop $\gamma$. That is, we have
\[ \gr(x) - \gr(x) \equiv (a,b,0) \in \frac{\tfrac 1 2 \Z \times \tfrac 1 2 \Z \times \Z}{\langle \Pi_0\rangle}, \]
where $(a,b)$ represents the homology class of $\gamma$ with respect to the basis $(\eta,\xi)$. Since $\gr(x) - \gr(x) \equiv (0,0,0)$, it follows that $(a,b,0)$ is a multiple of $\Pi_0 = (a_0, b_0, c_0)$. But since $c_0$ is a positive multiple of the winding number of $P$, by Lemma \ref{lem:periodic-grading}, which we assumed to be nonzero, it follows that $c_0 \neq 0$, and hence we must have that \[(a,b,0)=(0,0,0).\] This shows that $\gamma$ is nullhomologous, as desired.
\end{proof}

\subsection{Proof of Theorem \ref{thm:gensat}} 
Let $K\subset Y$ be a generalized satellite knot with companion $C\subset Z$. According to \S\ref{sec:generalized-satellite}, this means that we can write \[Y = -M_1 \cup_T M_2,\] where the satellite torus is identified with our fixed torus $T$, the knot $K$ is the image of a pattern knot $P\subset M_2$, and \[M_1 = -(Z\setminus \nu(C)).\] (Note that what we call $M_1$ here differs from what we called $M_1$ in \S\ref{sec:generalized-satellite} by an orientation reversal.) Let us then denote the multicurves $\Gamma(M_1)$ and $\Gamma(M_2,P)$ by $\Gamma_C$ and $\Gamma_P$ for short. These are curves in $\partial M_1$ and $\partial M_2$, but as in Proposition \ref{prop:curves-pairing} and Remark \ref{rmk:commen}, we may also view them as multicurves in $T_\bullet$, via the boundary identifications $T\to \partial M_i$. 

Proposition \ref{prop:curves-pairing} then tells us that
\[ \dim \HFK(Y, K) \ge i(\Gamma_P, \Gamma_C). \]
We wish to compare this to $\dim \HFK(Z, C)$, which according to \cite[Proposition~12]{hrw-properties} can be recovered from $\Gamma_C$ by taking the Floer homology with the single noncompact Lagrangian $\Gamma_\mu$ that is embedded and homologous to $\mu_C$. That is,
\[ \HFK(Z, C) \cong \HF(\Gamma_\mu, \Gamma_C). \]
(One may also view this as an application of Proposition \ref{prop:curves-pairing}, as it is straightforward to compute that the multicurve for the core curve in the solid torus is $\Gamma_\mu$.) Each component of $\Gamma_C$ is homologous to a multiple of the rational longitude of $M_1$, by \cite[Corollary~6.6]{hrw}. Since this longitude is dual to $\mu_C$, as noted in Remark \ref{rmk:dual}, it follows that no component of $\Gamma_C$ is homotopic or commensurable to $\Gamma_\mu$, and we have that
\[\dim \HF(\Gamma_\mu, \Gamma_C) = i(\Gamma_\mu, \Gamma_C).\]
Theorem \ref{thm:gensat} therefore follows immediately from the proposition below:

\begin{proposition}
If the pattern $P\subset M_2$ has nonzero winding number, then 
\[ i(\Gamma_P, \Gamma_C) \ge i(\Gamma_\mu, \Gamma_C). \]
\end{proposition}

\begin{proof}
We will adapt the proof of \cite[Theorem 7.15]{hrw}, which as a special case gives an inequality for Heegaard Floer homology under pinching. 

By Definition \ref{def:generalized-satellite}, we have that $H_2(M_2) = 0$, which implies as in \S\ref{sec:generalized-satellite} that \[H^1(M_2) \cong H_2(M_2,\partial M_2) \cong \Z.\] In particular, $b_1(M_2) = 1$, so Proposition \ref{prop:curve-homology-class} says that $\Gamma_P$ is homologous to a nonzero multiple of the rational longitude of $M_2$, which by Definition \ref{def:generalized-companion} is the meridian $\mu_C$. The assumption that $P$ has nonzero winding number implies by Proposition \ref{prop:noncompact} that each compact component of $\Gamma_P$ is nullhomologous. Therefore,  $\Gamma_P$ has at least one noncompact component, and the total homology class of these noncompact components is a nonzero multiple of $[\mu_C] = [\Gamma_\mu]$.

With this setup in hand, our strategy is to start with $\Gamma = \Gamma_P$ and transform it into $\Gamma = \Gamma_\mu$ via a sequence of simple moves that do not increase the intersection number with the fixed compact curve $\Gamma_C$, which we assume has been pulled tight in $T_\bullet$. These moves are:
\begin{enumerate}[label=(\roman*)]
    \item resolve a self-intersection of $\Gamma$ in an oriented way; \label{i:move1}
    \item replace two (oriented) immersed arcs with one, by connecting the terminal end of one to the initial end of the other in a neighborhood of the puncture; \label{i:move2}
    \item delete a component of $\Gamma$; \label{i:move3}
    \item pull $\Gamma$ tight in the complement of the puncture; and \label{i:move4}
    \item assuming $\Gamma$ is pulled tight, pull $\Gamma$ through the puncture at one nonstraight corner (see \cite[Figure 54]{hrw}). \label{i:move5}
\end{enumerate}
See Figure~\ref{fig:curve-moves} for illustrations of these moves, and Figure~\ref{fig:example-moves} for an example (we doubt this corresponds to an actual generalized satellite) of how these can be used to turn $\Gamma_P$ into $\Gamma_\mu$ without ever increasing the intersection number with $\Gamma_C$. 

\newcommand{\moveref}[1]{%
\begingroup%
\hypersetup{hidelinks}%
\ref{i:move#1}%
\endgroup%
}
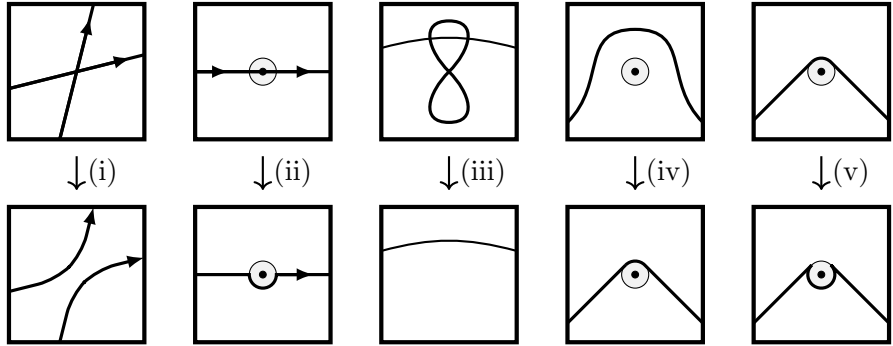
\begin{figure}
\begin{tikzpicture}[scale=0.9]
\foreach \i in {1,2,3,4,5} {
\draw[ultra thick] (2.75*\i,0) coordinate (lower\i) rectangle ++(2,2);
\draw[ultra thick] (lower\i) ++(0,3) coordinate (upper\i) rectangle ++(2,2);
\draw[->] (upper\i) ++(1,-0.25) to coordinate[midway] (arrow\i) ++(0,-0.5);
\node[right] at (arrow\i) {\moveref{\i}};
}
\begin{scope}
\clip (upper1) rectangle ++(2,2);
\draw[very thick] (upper1) ++ (0,0.75) coordinate (a) -- coordinate[pos=0.9] (b) ++(2,0.5);
\draw[very thick] (upper1) ++ (0.75,0) coordinate (c) -- coordinate[pos=0.9] (d) ++(0.5,2);
\draw[very thick, -latex] (a) -- (b);
\draw[very thick, -latex] (c) -- (d);
\end{scope}
\begin{scope}
\clip (lower1) rectangle ++(2,2);
\path (lower1) ++ (0,0.75) coordinate (a) -- coordinate[pos=0.35] (b) coordinate[pos=0.65] (c)  ++(2,0.5) coordinate (d);
\path (lower1) ++ (0.75,0) coordinate (e) -- coordinate[pos=0.35] (f) coordinate[pos=0.65] (g) ++(0.5,2) coordinate (h);
\draw[very thick,rounded corners=2.5mm,-latex] (a) -- (b) -- (g) -- (h);
\draw[very thick,rounded corners=2.5mm,-latex] (e) -- (f) -- (c) -- (d);
%\draw[very thick, -latex] (lower1) ++ (0.75,0) coordinate (c) to[bend left=40] ++ (1.25,1.25) coordinate (d);
\end{scope}
\begin{scope}
\clip (upper2) rectangle ++(2,2);
\draw[thin,fill=gray!10] (upper2) ++ (1,1) circle (0.2) coordinate (a);
\draw[thin,fill=black] (a) circle (0.05);
\draw[very thick] (upper2) ++(0,1) coordinate (b) -- ++(2,0);
\draw[very thick, -latex] (b) -- ++(0.5,0);
\draw[very thick, -latex] (b) -- ++(1.75,0);
\end{scope}
\begin{scope}
\clip (lower2) rectangle ++(2,2);
\draw[thin,fill=gray!10] (lower2) ++ (1,1) coordinate (a) circle (0.2);
\draw[thin,fill=black] (a) circle (0.05);
\draw[very thick] (lower2) ++(0,1) coordinate (b) -- ++(0.8,0) arc (180:360:0.2) -- ++(0.8,0);
\draw[very thick, -latex] (b) ++(1.25,0) -- ++(0.5,0);
\end{scope}
\begin{scope}
\clip (upper3) rectangle ++(2,2);
\draw[very thick,looseness=1.5] (upper3) ++(1,1) coordinate (b) to[out=45,in=0] ++(0,0.75) to[out=180,in=135] ++(0,-0.75) to[out=315,in=0] ++(0,-0.75) to[out=180,in=225] ++(0,0.75);
\draw[thick] (upper3) ++(0,1.35) to[bend left=15] ++(2,0);
\end{scope}
\begin{scope}
\clip (lower3) rectangle ++(2,2);
\draw[thick] (lower3) ++(0,1.35) to[bend left=15] ++(2,0);
\end{scope}
\begin{scope}
\clip (upper4) rectangle ++(2,2);
\draw[thin,fill=gray!10] (upper4) ++ (1,1) circle (0.2) coordinate (a);
\draw[thin,fill=black] (a) circle (0.05);
\draw[very thick,looseness=1.25] (upper4) ++ (0,0.25) to[out=45, in=180] ++(1,1.375) to[out=0,in=135] ++(1,-1.375);
\end{scope}
\begin{scope}
\clip (lower4) rectangle ++(2,2);
\draw[thin,fill=gray!10] (lower4) ++ (1,1) circle (0.2) coordinate (a);
\draw[thin,fill=black] (a) circle (0.05);
\draw[very thick] (a) ++(135:0.2) coordinate (b) -- ++(225:2);
\draw[very thick] (b) arc (135:45:0.2) -- ++(-45:2);
\end{scope}
\begin{scope}
\clip (upper5) rectangle ++(2,2);
\draw[thin,fill=gray!10] (upper5) ++ (1,1) circle (0.2) coordinate (a);
\draw[thin,fill=black] (a) circle (0.05);
\draw[very thick] (a) ++(135:0.2) coordinate (b) -- ++(225:2);
\draw[very thick] (b) arc (135:45:0.2) -- ++(-45:2);
\end{scope}
\begin{scope}
\clip (lower5) rectangle ++(2,2);
\draw[thin,fill=gray!10] (lower5) ++ (1,1) circle (0.2) coordinate (a);
\draw[thin,fill=black] (a) circle (0.05);
\draw[very thick] (a) ++(135:0.2) coordinate (b) -- ++(225:2);
\draw[very thick] (b) arc (135:405:0.2) -- ++(-45:2);
\end{scope}
\end{tikzpicture}
\caption{Local pictures of the moves we use to simplify the multicurve $\Gamma$.}\label{fig:curve-moves}
\end{figure}
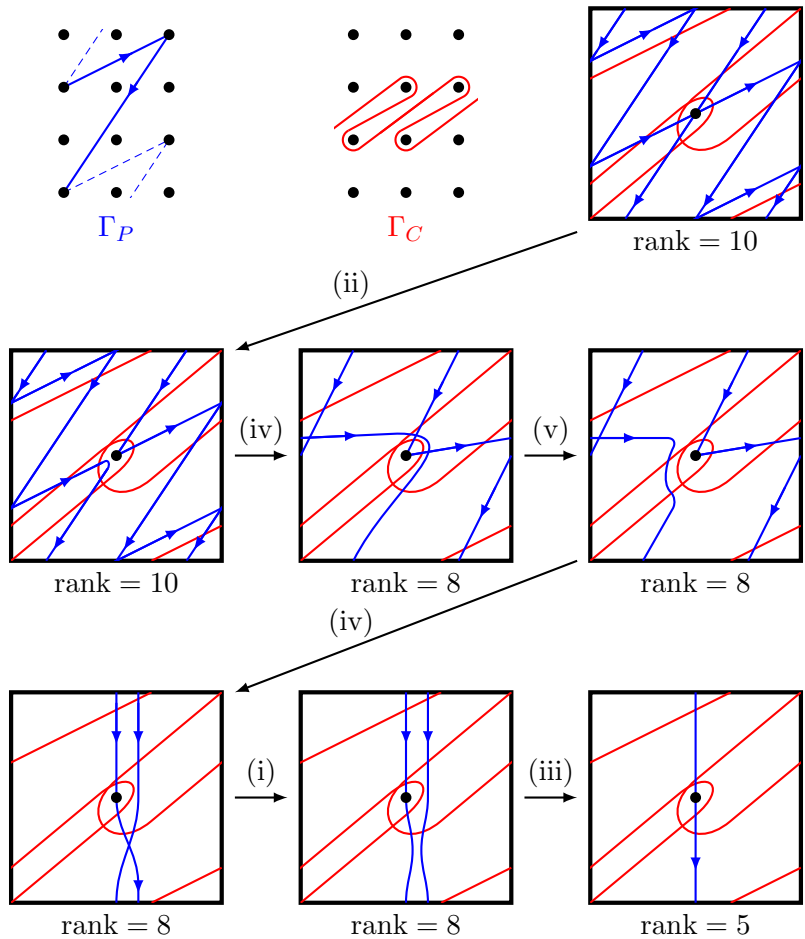
\begin{figure}
\begin{tikzpicture}[scale=0.7]
\begin{scope}
\node[blue,below] at (2,0.25) {$\Gamma_P$};
\clip (0.65,0.4) rectangle (3.35,3.6);
\draw[blue] (1,0.5) -- coordinate[pos=0.6] (b1) coordinate[pos=0.7] (a1) (3,3.5);
\draw[blue] (1,2.5) -- coordinate[pos=0.65] (a2) coordinate[pos=0.66] (b2) (3,3.5);
\draw[blue,thin,densely dashed] (1,2.5) -- ++(2,3) (1,0.5) -- ++(2,1) -- ++(-2,-3);
\foreach \i in {1,2} { \draw[blue,-latex] (a\i) -- (b\i); }
\foreach \i in {1,2,3} {
  \foreach \j in {0.5,1.5,2.5,3.5} {
    \draw[fill=black] (\i,\j) circle (0.25/3);
  }
}
\end{scope}
\begin{scope}[xshift=5.5cm]
\node[red,below] at (2,0.25) {$\Gamma_C$};
\clip (0.65,0.4) rectangle (3.35,3.6);
\begin{scope}[every path/.style={red}]
\foreach \i in {1,2} {
\draw (\i,1.5) ++ (118.57:0.2) coordinate (a);
\draw (\i,2.5) ++ (1,0) ++ (298.57:0.2) coordinate (b);
\draw (a) -- (b) arc (298.57:478.57:0.2) -- ++ (216.87:1.75);
\draw (a) arc (118.57:306.87:0.2) -- ++(36.87:1.75);
}
\end{scope}
\foreach \i in {1,2,3} {
  \foreach \j in {0.5,1.5,2.5,3.5} {
    \draw[fill=black] (\i,\j) circle (0.25/3);
  }
}
\end{scope}
\begin{scope}[xshift=11cm,scale=1/3]
\draw[ultra thick, fill=white] (0,0) coordinate (start) rectangle ++(12,12);
\node[below] at (6,0) {$\mathrm{rank}=10$};
\clip (start) rectangle ++(12,12);
\begin{scope}[every path/.style = {red}]
  \draw (start) ++ (8,12) -- ++(-12,-6);
  \draw (start) ++ (8,0) -- ++(4,2) ++ (-12,0) ;
  \draw (start) -- ++(6,5) to[out=39.8, in=39.8, looseness=4] ++(-5*12/61, 6*12/61) -- ++(-6,-5);
  \draw (start) ++ (12,12) -- ++(-6,-5) to[out=219.8, in=219.8, looseness=2.5] ++(10*12/61, -12*12/61) -- ++(6,5);
\end{scope}
\begin{scope}[every path/.style = {blue}]
\foreach \i in {-8,-4,0,4,8} {
  \draw (start) ++ (6,6) ++ (\i,0) ++(4,6) coordinate (a) -- coordinate[pos=0.2] (b) coordinate[pos=0.95] (c) ++(-8,-12) coordinate (d);
  \draw[-latex] (a) -- (b);
  \draw[-latex] (b) -- (c);
}
\foreach \i in {-6,0,6} {
  \draw (start) ++ (6,6) ++ (0,\i) ++(6,3) coordinate (a) -- coordinate[pos=0.2] (b) coordinate[pos=0.7] (c) ++(-12,-6) coordinate (d);
  \draw[-latex] (d) -- (c);
  \draw[-latex] (c) -- (b);
}
\end{scope}
\draw[fill=black] (6,6) circle (0.25);
\end{scope}
\draw[-latex] (10.75,-0.25) -- node[above,pos=0.66] {\moveref{2}} (4.25,-2.5);
\begin{scope}[yshift=-6.5cm, scale=1/3]
\draw[ultra thick, fill=white] (0,0) coordinate (start) rectangle ++(12,12);
\node[below] at (6,0) {$\mathrm{rank}=10$};
\clip (start) rectangle ++(12,12);
\begin{scope}[every path/.style = {red}]
  \draw (start) ++ (8,12) -- ++(-12,-6);
  \draw (start) ++ (8,0) -- ++(4,2) ++ (-12,0) ;
  \draw (start) -- ++(6,5) to[out=39.8, in=39.8, looseness=4] ++(-5*12/61, 6*12/61) -- ++(-6,-5);
  \draw (start) ++ (12,12) -- ++(-6,-5) to[out=219.8, in=219.8, looseness=2.5] ++(10*12/61, -12*12/61) -- ++(6,5);
\end{scope}
\begin{scope}[every path/.style = {blue}]
\foreach \i in {-8,-4,0,4,8} {
  \draw (start) ++ (6,6) ++ (\i,0) ++(4,6) coordinate (a) -- coordinate[pos=0.2] (b) coordinate[pos=0.95] (c) ++(-8,-12) coordinate (d);
  \draw[-latex] (a) -- (b);
  \draw[-latex] (b) -- (c);
}
\foreach \i in {-6,0,6} {
  \draw (start) ++ (6,6) ++ (0,\i) ++(6,3) coordinate (a) -- coordinate[pos=0.2] (b) coordinate[pos=0.7] (c) ++(-12,-6) coordinate (d);
  \draw[-latex] (d) -- (c);
  \draw[-latex] (c) -- (b);
}
% merge two blue arcs
\draw[white,fill=white] (start) ++ (5.75,5.75) circle (0.5);
\path (start) ++(0,3) ++(5,2.5) -- ++(1/3,1/6) coordinate (a);
\draw[blue] (start) ++(2,0) ++(3,4.5) -- ++(0.5,0.75) to[out=56.31,in=27.565,looseness=1.5] (a) -- ++(-0.5,-0.25);
\end{scope}
\draw[fill=black] (6,6) circle (0.25);
\end{scope}
\draw[-latex] (4.25,-4.5) -- node[above,midway] {\moveref{4}} ++(1,0);
\begin{scope}[xshift=5.5cm,yshift=-6.5cm, scale=1/3]
\draw[ultra thick, fill=white] (0,0) coordinate (start) rectangle ++(12,12);
\node[below] at (6,0) {$\mathrm{rank}=8$};
\clip (start) rectangle ++(12,12);
\begin{scope}[every path/.style = {red}]
  \draw (start) ++ (8,12) -- ++(-12,-6);
  \draw (start) ++ (8,0) -- ++(4,2) ++ (-12,0) ;
  \draw (start) -- ++(6,5) to[out=39.8, in=39.8, looseness=4] ++(-5*12/61, 6*12/61) -- ++(-6,-5);
  \draw (start) ++ (12,12) -- ++(-6,-5) to[out=219.8, in=219.8, looseness=2.5] ++(10*12/61, -12*12/61) -- ++(6,5);
\end{scope}
\begin{scope}[every path/.style = {blue}]
\draw (start) ++ (6,6) coordinate (a1) -- coordinate[pos=0.6] (b1) ++(6,1);
\draw (start) ++ (0,7) coordinate (a2) -- coordinate[pos=0.6] (b2) ++(5.5,0.25) to[out=0,in=63.435,looseness=1.5] ++(-2.5,-7.25);
\draw (start) ++ (3,12) -- coordinate[pos=0.3] (a3) coordinate[pos=0.4] (b3) ++(-3,-6) ++(12,0) -- ++(-3,-6) coordinate[pos=0.3] (a4) coordinate[pos=0.6] (b4)  ++ (0,12) -- ++(-3,-6) coordinate[pos=0.3] (a5) coordinate[pos=0.5] (b5);
\foreach \i in {1,2,3,4,5} { \draw[-latex] (a\i) -- (b\i); }
\end{scope}
\draw[fill=black] (6,6) circle (0.25);
\end{scope}
\draw[-latex] (9.75,-4.5) -- node[above,midway] {\moveref{5}} ++(1,0);
\begin{scope}[xshift=11cm,yshift=-6.5cm, scale=1/3]
\draw[ultra thick, fill=white] (0,0) coordinate (start) rectangle ++(12,12);
\node[below] at (6,0) {$\mathrm{rank}=8$};
\clip (start) rectangle ++(12,12);
\begin{scope}[every path/.style = {red}]
  \draw (start) ++ (8,12) -- ++(-12,-6);
  \draw (start) ++ (8,0) -- ++(4,2) ++ (-12,0) ;
  \draw (start) -- ++(6,5) to[out=39.8, in=39.8, looseness=4] ++(-5*12/61, 6*12/61) -- ++(-6,-5);
  \draw (start) ++ (12,12) -- ++(-6,-5) to[out=219.8, in=219.8, looseness=2.5] ++(10*12/61, -12*12/61) -- ++(6,5);
\end{scope}
\begin{scope}[every path/.style = {blue}]
\draw (start) ++ (6,6) coordinate (a1) -- coordinate[pos=0.6] (b1) ++(6,1);
\draw[rounded corners] (start) ++ (0,7) coordinate (a2) -- coordinate[pos=0.9] (b2) ++(3,0) -- ++(2,0) to[out=243.435,in=135,looseness=1] ++(0,-3.5) -- ++(-2,-3.5);
\draw (start) ++ (3,12) -- coordinate[pos=0.3] (a3) coordinate[pos=0.4] (b3) ++(-3,-6) ++(12,0) -- ++(-3,-6) coordinate[pos=0.3] (a4) coordinate[pos=0.6] (b4)  ++ (0,12) -- ++(-3,-6) coordinate[pos=0.3] (a5) coordinate[pos=0.5] (b5);
\foreach \i in {1,2,3,4,5} { \draw[-latex] (a\i) -- (b\i); }
\end{scope}
\draw[fill=black] (6,6) circle (0.25);
\end{scope}
\draw[-latex] (10.75,-6.5) -- node[above,pos=0.66] {\moveref{4}} (4.25,-9);
\begin{scope}[yshift=-13cm, scale=1/3]
\draw[ultra thick, fill=white] (0,0) coordinate (start) rectangle ++(12,12);
\node[below] at (6,0) {$\mathrm{rank}=8$};
\clip (start) rectangle ++(12,12);
\begin{scope}[every path/.style = {red}]
  \draw (start) ++ (8,12) -- ++(-12,-6);
  \draw (start) ++ (8,0) -- ++(4,2) ++ (-12,0) ;
  \draw (start) -- ++(6,5) to[out=39.8, in=39.8, looseness=4] ++(-5*12/61, 6*12/61) -- ++(-6,-5);
  \draw (start) ++ (12,12) -- ++(-6,-5) to[out=219.8, in=219.8, looseness=2.5] ++(10*12/61, -12*12/61) -- ++(6,5);
\end{scope}
\begin{scope}[every path/.style = {blue}]
\draw (start) ++ (6,6) to[out=270,in=90] ++(1.25,-5) -- coordinate[pos=0] (a1) coordinate[pos=0.5] (b1) ++(0,-1);
\draw (start) ++ (6,6) ++ (1.25,6) -- coordinate[pos=0.1] (a2) coordinate[pos=0.5] (b2) ++(0,-6) to[out=270,in=90] ++(-1.25,-6);
\draw (start) ++ (6,6) ++ (0,6) -- coordinate[pos=0.1] (a3) coordinate[pos=0.5] (b3) ++(0,-6);
\foreach \i in {1,2,3} { \draw[-latex] (a\i) -- (b\i); }
\end{scope}
\draw[fill=black] (6,6) circle (0.25);
\end{scope}
\draw[-latex] (4.25,-11) -- node[above,midway] {\moveref{1}} ++(1,0);
\begin{scope}[xshift=5.5cm,yshift=-13cm, scale=1/3]
\draw[ultra thick, fill=white] (0,0) coordinate (start) rectangle ++(12,12);
\node[below] at (6,0) {$\mathrm{rank}=8$};
\clip (start) rectangle ++(12,12);
\begin{scope}[every path/.style = {red}]
  \draw (start) ++ (8,12) -- ++(-12,-6);
  \draw (start) ++ (8,0) -- ++(4,2) ++ (-12,0) ;
  \draw (start) -- ++(6,5) to[out=39.8, in=39.8, looseness=4] ++(-5*12/61, 6*12/61) -- ++(-6,-5);
  \draw (start) ++ (12,12) -- ++(-6,-5) to[out=219.8, in=219.8, looseness=2.5] ++(10*12/61, -12*12/61) -- ++(6,5);
\end{scope}
\begin{scope}[every path/.style = {blue}]
\draw (start) ++ (6,6) to[out=270,in=90] ++ (0.375,-3) to[out=270,in=90] ++ (-0.375,-3) ++(0,12) -- coordinate[pos=0.4] (a1) coordinate[pos=0.5] (b1) ++ (0,-6);
\draw (start) ++ (6,6) ++(1.25,6) -- coordinate[pos=0.4] (a2) coordinate[pos=0.5] (b2) ++(0,-6) to[out=270,in=90] ++(-0.375,-3) to[out=270,in=90] ++(0.375,-3);
\foreach \i in {1,2} { \draw[-latex] (a\i) -- (b\i); }
\end{scope}
\draw[fill=black] (6,6) circle (0.25);
\end{scope}
\draw[-latex] (9.75,-11) -- node[above,midway] {\moveref{3}} ++(1,0);
\begin{scope}[xshift=11cm,yshift=-13cm, scale=1/3]
\draw[ultra thick, fill=white] (0,0) coordinate (start) rectangle ++(12,12);
\node[below] at (6,0) {$\mathrm{rank}=5$};
\clip (start) rectangle ++(12,12);
\begin{scope}[every path/.style = {red}]
  \draw (start) ++ (8,12) -- ++(-12,-6);
  \draw (start) ++ (8,0) -- ++(4,2) ++ (-12,0) ;
  \draw (start) -- ++(6,5) to[out=39.8, in=39.8, looseness=4] ++(-5*12/61, 6*12/61) -- ++(-6,-5);
  \draw (start) ++ (12,12) -- ++(-6,-5) to[out=219.8, in=219.8, looseness=2.5] ++(10*12/61, -12*12/61) -- ++(6,5);
\end{scope}
\begin{scope}[every path/.style = {blue}]
\draw (start) ++ (6,12) -- coordinate[pos=0.8] (a1)  coordinate[pos=0.85] (b1) ++(0,-12);
\foreach \i in {1} { \draw[-latex] (a\i) -- (b\i); }
\end{scope}
\draw[fill=black] (6,6) circle (0.25);
\end{scope}
\end{tikzpicture}
\caption{Taking a pair of curves $\Gamma_P$ and $\Gamma_C$, shown at the top in the universal cover of $T_{\bullet}$, and then applying a sequence of the moves (i)-(v) in order to turn $\Gamma_P$ into $\Gamma_\mu$ without ever increasing the intersection number with $\Gamma_C$.}\label{fig:example-moves}
\end{figure}

Move \moveref{1} does not change the intersection number, since it only modifies $\Gamma$ in a small neighborhood of the self-intersection point (we can assume these do not coincide with intersections with $\Gamma_C$). Move \moveref{2}, which does not occur in \cite{hrw}, is similar. Clearly \moveref{3} cannot increase the intersection number with $\Gamma_C$, and \moveref{4} does not increase intersection number since pulling tight realizes minimal intersection within a homotopy class of curves.

Move \moveref{5} also cannot increase intersection number, but this is somewhat subtle: if a corner of $\Gamma_C$ lies within a corner of $\Gamma$ as in \cite[Figure 55(a)]{hrw}, then pulling that corner of $\Gamma$ through the peg introduces two new intersection points.  However, in this case we can argue as in \cite{hrw} that for each such corner of $\Gamma_C$, there must be another corner at which this move eliminates two intersection points. (Note that this argument fails if the relevant corner of $\Gamma$ is a straight corner.) 

In addition to not increasing intersection number with $\Gamma_C$, we also note that moves \moveref{1}, \moveref{2}, \moveref{4}, and \moveref{5} do not change the homology class of $\Gamma$.

In carrying out our strategy, let us first consider the case that $\Gamma_P$ consists of a single immersed arc. By repeatedly applying moves \moveref{4} and \moveref{5}, we can replace $\Gamma_P$ with a pulled tight, immersed arc $\Gamma$ for which all corners are straight.  Since the total homology class of $\Gamma_P$ (and thus of $\Gamma$) is a nonzero multiple of $[\Gamma_\mu]$, it follows that $\Gamma$ lies in a neighborhood of the arc $\Gamma_\mu$. By resolving all self-intersection points via move \moveref{1} and pulling tight, we obtain a copy of $\Gamma_\mu$ possibly along with some compact components. Any compact components can be deleted by move \moveref{3}, so we are done.

It remains to consider the case that $\Gamma_P$ contains more than a single immersed arc. In this case, we first delete all compact components via move  \moveref{3}. These components were  nullhomologous, so the resulting curve $\Gamma$ still has homology class equal to a nonzero multiple of $[\Gamma_\mu]$. We then apply move \moveref{2} repeatedly until there is a single arc component, reducing to the previous case. 
\end{proof}

\begin{corollary}
\label{cor:satcover}
Let $K\subset S^3$ be a fibered satellite knot with companion $C$, whose pattern has winding number $w$. Suppose that $n$ is a natural number relatively prime to $w$ such that $\Sigma_n(K)$ is a rational homology sphere. Then \[\dim\HFK(\Sigma_n(C),C_n)\leq \dim\HFK(\Sigma_n(K),K_n).\]
\end{corollary}

\begin{proof}
Proposition~\ref{prop:branched-cover-satellite} says that $K_n \subset \Sigma_n(K)$ is a generalized satellite knot, with companion knot $C_n \subset \Sigma_n(C)$, and that its pattern  has also has  winding number $w$.  Note that $w$ is an integer since it is the winding number of a satellite in $S^3$, and moreover it is strictly positive by Lemma~\ref{lem:fibered-satellite} and the assumption that $K$ is fibered.  Thus the desired inequality is precisely Theorem~\ref{thm:gensat}.
\end{proof}

\section{Finitely many fibered predecessors and Gromov norm} \label{sec:main-proof}

Theorems \ref{thm:main} and \ref{thm:volume} will follow from our main technical result, Theorem \ref{thm:hfk-j-finite}, restated here:

\begin{usetheoremcounterof} {thm:hfk-j-finite}
Given natural numbers $M$ and $g$, there are only finitely many genus-$g$ fibered knots $K\subset S^3$ which satisfy \[\dim\HFK(\Sigma_{n_i}(K),K_{n_i})\leq M^{n_i}\] for some increasing sequence $\{n_i\}_{i\in\mathbb{N}}$ of primes. Moreover, \begin{equation}\label{eq:gromov}\lVert S^3\setminus K \rVert \leq \tfrac{3\pi}{v_3}(2g-1)\log(M)\end{equation} for each such knot $K$.
\end{usetheoremcounterof}

In proving the part of this theorem about Gromov norm, it will be helpful to recall that $\lVert S^3\setminus K\rVert$ is given by taking the sum of the volumes of the hyperbolic pieces in the JSJ decomposition of the knot complement, and then dividing this sum by $v_3$ \cite{soma}. In particular, this norm vanishes if this complement is Seifert fibered. The key result we will use is then due to Kojima and McShane:

\begin{theorem}[{\cite[Theorem 1.1]{kojima-mcshane}}] \label{thm:kj}
If $\psi$ is a pseudo-Anosov homeomorphism of a compact orientable surface $\Sigma$ with boundary, then the volume of the interior of its mapping torus satisfies \[\mathrm{vol}(\inr M_\psi) \leq 3\pi|\chi(\Sigma)|\log(\lambda(\psi)),\] where $\lambda(\psi)$ is the dilatation of $\psi$.
\end{theorem}

\begin{proof}[Proof of Theorem \ref{thm:hfk-j-finite}]
Let $\mathcal{S}_{M,g}$ be the set of genus-$g$ fibered knots $K\subset S^3$ such that \[\dim\HFK(\Sigma_{n_i}(K),K_{n_i})\leq M^{n_i}\] for some increasing sequence $\{n_i\}_{i\in\mathbb{N}}$ of  primes. Let us fix $M\in \mathbb{N}$ and prove by induction on $g$ that $\mathcal{S}_{M,g}$ is finite and that any knot in $\mathcal{S}_{M,g}$ satisfies the Gromov norm inequality \eqref{eq:gromov}, for all $g\in \mathbb{N}$.

For the base case $g=1$, recall that the only genus-1 fibered knots in $S^3$ are the trefoils and the figure-eight. This shows that $S_{M,1}$ is finite. For the trefoils, the left hand side of \eqref{eq:gromov} is zero since trefoil complements are Seifert fibered. If the figure-eight $4_1$ is in $S_{M,1}$ then $M\geq 2$, since $M=1$ would imply that the lift $(4_1)_{n_i}$ is an unknot in its corresponding branched cover. Then, since $4_1$ has hyperbolic complement with volume $2v_3$, we see explicitly that \[\lVert S^3\setminus 4_1\rVert = \tfrac{1}{v_3}\mathrm{vol}(S^3\setminus 4_1) = 2< \tfrac{3\pi}{v_3}\log(2) \leq \tfrac{3\pi}{v_3}(2g-1)\log(M),\] and hence \eqref{eq:gromov} is satisfied, completing the base case.
 
Now let $g\geq 2$. Suppose that $K\in \mathcal{S}_{M,g}$, and let $h:\Sigma\to \Sigma$ be its monodromy. We know from geometrization for Haken manifolds that $K$ is either a torus knot, a hyperbolic knot, or a satellite knot \cite[Corollary 2.5]{thurston-kleinian}. We will examine these cases in turn, proving in each case that there are only finitely many possibilities for $K$ and that $K$ satisfies \eqref{eq:gromov}.

\vspace{1em}\noindent
\underline{Case I: $K$ is a torus knot.} There are finitely many possibilities for $K$ since there are only finitely many torus knots of a given genus: the torus knot $T(p,q)$ has genus \[g(T(p,q)) = \frac{(|p|-1)(|q|-1)}{2},\] where $|p|,|q| \geq 2$. Moreover, the left hand side of \eqref{eq:gromov} is zero in this case since torus knots have Seifert fibered exteriors.

\vspace{1em}\noindent
\underline{Case II: $K$ is hyperbolic.} In this case, $h$ is freely isotopic to a pseudo-Anosov homeomorphism $\psi:\Sigma\to \Sigma$. To prove that there are only finitely many possibilities for $K$, it suffices to bound the dilatation $\lambda(\psi)$ independently of $K$, since (1) up to conjugation, there are only finitely many pseudo-Anosov homeomorphisms of the surface $\Sigma$ with dilatation less than a fixed constant \cite{ivanov-bounded-dilatation}, and (2) per \cite{gordon-luecke-complement}, the knot $K$ is determined by its complement \[S^3\setminus \nu(K)\cong M_\psi, \] which is determined by the conjugacy class of $\psi$. 

A well-known result of Thurston \cite{thurston-surfaces} (see also \cite{flp,ivanov-entropy} and \cite[Lemma 2.4]{bs-ribbon}) says that the dilatation of $\psi$ is determined by the numbers of fixed points of the iterates of $\psi$, \begin{equation}\label{eq:ni-fixed}\lambda(\psi) = \lim_{n\to\infty}\big(\#\textrm{Fix}(\psi^n)\big)^{\frac{1}{n}}= \lim_{i\to\infty}\big(\#\textrm{Fix}(\psi^{n_i})\big)^{\frac{1}{n_i}}.\end{equation}
Note that the lift $K_{n_i}$ is fibered with monodromy freely isotopic to $\psi^{n_i}:\Sigma\to\Sigma$. It follows that \[\#\textrm{Fix}(\psi^{n_i}) \leq \dim\HFK(\Sigma_{n_i}(K),K_{n_i})\leq M^{n_i},\] by the work of Ni \cite{ni-fixed} and Ghiggini--Spano \cite{ghiggini-spano}. Combined with \eqref{eq:ni-fixed}, this shows that $\psi$ has dilatation bounded independently of $K$, \[\lambda(\psi)\leq M,\] as desired, proving the finiteness of such $K$. Moreover, Theorem \ref{thm:kj} then implies that \[ \lVert S^3\setminus K \rVert = \tfrac{1}{v_3}\mathrm{vol}(S^3\setminus K) \leq \tfrac{3\pi}{v_3}(2g-1)\log(\lambda(\psi))\leq \tfrac{3\pi}{v_3}(2g-1)\log(M),\] establishing \eqref{eq:gromov} in this case.

\vspace{1em}\noindent
\underline{Case III: $K$ is a satellite knot.} In this case, $h$ is freely isotopic to a map $\phi:\Sigma\to\Sigma$ in Nielsen--Thurston form, with nonempty reducing set $\Gamma$. Let us adopt the notation from \S\ref{sec:reducing-curves}. In particular, let $\varphi_0:\Sigma_0\to\Sigma_0$ be the restriction of $\varphi$ to the outermost component. Then $\varphi_0$ is freely isotopic to either a periodic map or a pseudo-Anosov map. We will address these two subcases in turn.

\underline{Subcase III.1: $\varphi_0$ is freely isotopic to a periodic map.}
In this case, \cite[Proposition 2.3]{bs-traces} says that $K$ is either a connected sum of at least two nontrivial fibered knots, \[K= K_1\#\dots \# K_m,\] or $K$ is the $(p,q)$-cable of some knot $C\subset S^3$ with $|q|\geq 2$, in which case the genus and number of boundary components of $\Sigma_0$ are given by \begin{equation}\label{eq:topology}g(\Sigma_0) = g(T(p,q)) \quad \textrm{and} \quad |\partial\Sigma_0| = |q|+1.\end{equation}  

In the connected sum case, we can view $K$ as a satellite with companion $K_j$ and winding number $1$ pattern, for each $j=1,\dots,m$. Then Corollary \ref{cor:satcover} implies that \[\dim\HFK(\Sigma_n(K_j),(K_j)_n)\leq \dim\HFK(\Sigma_n(K),K_n)\] for every $n\geq 1$ for which $\Sigma_n(K)$ is a rational homology sphere. It follows from Proposition \ref{prop:resultants} that the latter condition is met, and in fact that $\Sigma_n(K)$ is a $\Z/2$-homology sphere, for all primes $n$ which are not contained in the finite set $S_{K,2}$.
In particular, we may assume without loss of generality that the inequality above, and hence \[\dim\HFK(\Sigma_{n_i}(K_j),(K_j)_{n_i})\leq M^{n_i},\] holds for all of the primes $n_i$.  Since $g(K_j)<g$, we conclude by induction that there are finitely many possibilities for each $K_j$. It then follows that there are only finitely many possibilities for $K$, as desired. We also have by induction that \eqref{eq:gromov} holds for each $K_j$. Since \[\lVert S^3\setminus K \rVert = \lVert S^3\setminus K_1 \rVert + \dots + \lVert S^3\setminus K_m \rVert\] and $g = g(K_1)+ \dots + g(K_m)$, it follows immediately that \eqref{eq:gromov}  holds for $K$ as well.

In the cable case, note that there are finitely many possibilities for the topology of $\Sigma_0$ since it's a subsurface of $\Sigma$. This then implies that there are finitely many possibilities for $(p,q)$ by \eqref{eq:topology}, so it suffices to show that there are only finitely many possibilities for the companion knot $C$ for each such pair $(p,q)$. The winding number of the pattern in this case is $|q|\geq 2$. Corollary \ref{cor:satcover} then implies that \[\dim\HFK(\Sigma_n(C),C_n)\leq \dim\HFK(\Sigma_n(K),K_n)\] for all $n$ relatively prime to $|q|$ for which  $\Sigma_n(K)$ is a rational homology sphere. These two conditions are met in particular by all but the finitely many primes $n$ which either divide $|q|$ or are contained in the finite set $S_{K,2}$ of Proposition \ref{prop:resultants}, so we may assume without loss of generality that it holds for all of the primes $n_i$. Then we have  \[\dim\HFK(\Sigma_{n_i}(C),C_{n_i})\leq M^{n_i}\] for all $i$. Since \[g=g(K) = |q|g(C) + g(T(p,q))>g(C),\] we conclude by induction that there are only finitely many possibilities for $C$ and hence for $K$, as desired. We also have by induction that \eqref{eq:gromov} holds for $C$. Since $M_{\varphi_0}$ is Seifert fibered in this case, we have that $\lVert S^3\setminus K \rVert = \lVert S^3\setminus C \rVert$, and hence that \[\lVert S^3\setminus K \rVert =\lVert S^3\setminus C \rVert\leq \tfrac{3\pi}{v_3}(2g(C)-1)\log(M) \leq \tfrac{3\pi}{v_3}(2g-1)\log(M),\] proving \eqref{eq:gromov} for $K$ as well.

\underline{Subcase III.2: $\varphi_0$ is freely isotopic to a pseudo-Anosov map $\psi_0$.}
Assume the notation of \S\ref{sec:reducing-curves}.  In order to establish that there are only finitely many possibilities for $K$, it suffices by Theorem \ref{thm:determine-C} to prove that there are only finitely many possibilities for the companion knots $C^1,\dots,C^k$ and the conjugacy class of $\psi_0$. According to Lemma \ref{lem:sat-longitude}, the companion $C^j$ has pattern $P^j$ with winding number $m_j\geq 1$. This winding number is bounded above by $g$, since $C^j$ is nontrivial and  \[g=m_jg(C^j)+ g(P^j(U)) \geq m_jg(C^j), \] where $P^j(U)\subset S^3$ is the knot obtained by applying the pattern $P^j$ to the unknot $U$. In particular, all but finitely many $n_i$ are relatively prime to $m_j$, so we may assume without loss of generality that all of them are, in which case Corollary \ref{cor:satcover} implies that \[\dim\HFK(\Sigma_{n_i}(C^j),(C^j)_{n_i})\leq \dim\HFK(\Sigma_{n_i}(K),K_{n_i}) \leq M^{n_i}\] for all $i$. Hence, since $g(C^j)<g$, there are only finitely many possibilities for each $C^j$, by induction. Moreover, we have that each $C^j$ satisfies the Gromov norm bound \eqref{eq:gromov}.

It remains to show that there are only finitely many possibilities for the conjugacy class of $\psi_0$, and since there are only finitely many possibilities for the topology of $\Sigma_0$, it suffices for this to prove that the dilatation $\lambda(\psi_0)$ is bounded above independently of $K$, just as in the case that $K$ was hyperbolic. We bound this dilatation in almost exactly the same manner as in that case. Namely, the number of fixed points of $\psi_0^n$ is bounded above by the Nielsen number of $\varphi^{n_i}$, which is bounded above by the knot Floer homology of the lift $K_{n_i}$ by \cite{ni-fixed,ghiggini-spano}, and so we have that \[\#\textrm{Fix}(\psi_0^{n_i})\leq N(\varphi^{n_i})\leq \dim\HFK(\Sigma_{n_i}(K),K_{n_i})\leq M^{n_i},\] and hence that \begin{equation}\label{eq:psi-dil}\lambda(\psi_0) = \lim_{i\to \infty} \big(\#\textrm{Fix}(\psi_0^{n_i})\big)^{\frac{1}{n_i}}\leq M,\end{equation} as desired. This completes the proof that there are only finitely many $K$.

For the Gromov norm bound, note that Theorem \ref{thm:kj} and \eqref{eq:psi-dil} give the inequalities \[\lVert \inr M_{\psi_0}\rVert = \tfrac{1}{v_3}\mathrm{vol}(\inr M_{\psi_0}) \leq \tfrac{3\pi}{v_3}|\chi(\Sigma_0)|\log(\lambda(\psi_0))\leq \tfrac{3\pi}{v_3}|\chi(\Sigma_0)|\log(M).\] Now, the fiber surface $\Sigma$ for $K$ is the union of $\Sigma_0$ with $m_j$ copies of the fiber surface of $C^j$, for each $j=1,\dots,k$, as described in \S\ref{sec:reducing-curves}. It follows that the negative Euler characteristic of $\Sigma$ is given by \[2g-1 = -\chi(\Sigma_0) + m_1(2g(C^1)-1) + \dots + m_k(2g(C^k)-1),\] which implies that \[2g-1 \geq |\chi(\Sigma_0)| + (2g(C^1)-1) + \dots + (2g(C^k)-1).\] Then, since \[\lVert S^3\setminus K\rVert =  \lVert \inr M_{\psi_0}\rVert + \lVert S^3\setminus C^1\rVert + \dots + \lVert S^3\setminus C^k\rVert,\] and  \eqref{eq:gromov} holds for each $C^j$ by induction, it follows that \eqref{eq:gromov} holds for $K$ as well. 
\end{proof}

We may now prove Theorem \ref{thm:main} and \ref{thm:volume} at the same time:

\begin{proof}[Proofs of Theorems \ref{thm:main} and  \ref{thm:volume}]
Let $K\subset S^3$ be a knot of genus $g$ and arc index $\delta$. Suppose that $J\subset S^3$ is a fibered knot with $J\leq K$. Then
\[ g(J) = \deg \Delta_J(t) \leq \deg \Delta_K(t) \leq g, \]
where the first two relations hold because $J$ is fibered and because $\Delta_J(t)$ divides $\Delta_K(t)$ \cite{gilmer}, respectively.  Since both theorems hold automatically for the unknot, it thus suffices to show for each $k=1,\dots,g$ that there are only finitely many such $J$ of genus $k$, proving Theorem \ref{thm:main}; and that $J$ satisfies \[\lVert S^3\setminus J\rVert \leq \tfrac{3\pi}{v_3}(2g-1)\log(\delta!),\] proving Theorem \ref{thm:volume}.

Theorem~\ref{thm:hfk-cover} provides a finite set $S$ of prime numbers such that 
\[\dim \HFK(\Sigma_n(J),J_n) \leq \dim\HFK(\Sigma_n(K), K_n)\] as long as $n$ is not a multiple of any prime in $S$, for any  $J\leq K$ as above. Enumerating the primes $n > \max(S)$ in order as $n_1,n_2,n_3,\dots$, we then have in particular that
\[ \dim \hfkhat(\Sigma_{n_i}(J), J_{n_i}) \leq \dim \hfkhat(\Sigma_{n_i}(K), K_{n_i}). \]
At the same time, we can appeal to \cite[Lemma~2.1]{bs-ribbon} (which uses a construction of multi-pointed Heegaard diagrams for branched cyclic covers of knots in $S^3$, due to Levine \cite{levine-cyclic}) to assert that
\[ \dim \hfkhat(\Sigma_{n_i}(K), K_{n_i}) \leq \frac{(\delta!)^{n_i}}{2^{\delta-1}} \leq (\delta!)^{n_i}. \]
Combining these inequalities, it follows that
\[ \dim \hfkhat(\Sigma_{n_i}(J), J_{n_i}) \leq (\delta!)^{n_i}, \] for all $n_i$.
Then Theorem~\ref{thm:hfk-j-finite} says that there are only finitely many genus-$k$ fibered knots $J$ that satisfy this condition for each $k=1,\dots,g$, and that each such $J$ satisfies \[\lVert S^3\setminus J\rVert \leq \tfrac{3\pi}{v_3}(2k-1)\log(\delta!)\leq \tfrac{3\pi}{v_3}(2g-1)\log(\delta!),\] proving both theorems.
\end{proof}

\bibliographystyle{myalpha}
\bibliography{References}

\end{document}